\documentclass[10pt]{article}

\usepackage{definitions}
\usepackage{graphicx}
\usepackage{float}
\usepackage{subcaption}
\usepackage[numbers]{natbib}

\usepackage{amsmath, amsthm, amssymb, amsfonts, mathtools}
\usepackage{mathrsfs}
\usepackage{bm}
\numberwithin{equation}{section}
\usepackage{enumitem}
\usepackage{mathtools}

\usepackage{graphicx}
\usepackage{booktabs}
\usepackage{enumitem}
\setlist[itemize]{leftmargin=1.5em}
\setlist[enumerate]{leftmargin=1.5em}

\usepackage{xcolor}
\definecolor{Violet Purple}{RGB}{128, 29, 174}
\usepackage{hyperref}
\hypersetup{
  colorlinks=true,
  linkcolor=blue!60!black,
  citecolor=blue!60!black,
  urlcolor=blue!60!black
}

\usepackage{titlesec}
\titleformat{\section}{\normalfont\Large\bfseries}{\thesection}{0.8em}{}
\titleformat{\subsection}{\normalfont\large\bfseries}{\thesubsection}{0.8em}{}
\titleformat{\subsubsection}{\normalfont\normalsize\bfseries}{\thesubsubsection}{0.6em}{}
\titlespacing*{\section}{0pt}{1.2ex plus .2ex}{0.8ex}
\titlespacing*{\subsection}{0pt}{1.0ex plus .2ex}{0.6ex}
\titlespacing*{\subsubsection}{0pt}{0.8ex plus .2ex}{0.4ex}

\newtheorem{theorem}{Theorem}[section]
\newtheorem{lemma}[theorem]{Lemma}
\newtheorem{proposition}[theorem]{Proposition}

\newtheorem{remark}[theorem]{Remark}
\theoremstyle{definition}

\newcommand{\R}{\mathbb{R}}
\newcommand{\E}{\mathbb{E}}
\newcommand{\N}{\mathbb{N}}
\newcommand{\norm}[1]{\left\lVert #1 \right\rVert}
\newcommand{\cFMM}{\mathcal{F}_{m,M}}

\allowdisplaybreaks

\begin{document}

\title{Generalization of Silver Stepsize Schedule to Stochastic Optimization}

\author{Luwei Bai\footnote{School of Computing and Augmented Intelligence, Arizona State University (luweibai@asu.edu).} \and Yang Zeng\footnote{School of Computing and Augmented Intelligence, Arizona State University (yzeng87@asu.edu).} \and Baoyu Zhou\footnote{School of Computing and Augmented Intelligence, Arizona State University (baoyu.zhou@asu.edu).}}
\date{}
\maketitle

\begin{abstract}
  This work introduces a two-step stepsize schedule for stochastic gradient methods minimizing smooth strongly convex functions.  
We consider the setting where only stochastic gradient approximations, which are unbiased, of bounded variance, and supported on a finite set, are accessible. 
When the variance bound is relatively smaller than a ratio of the initial optimality gap, the proposed stepsize schedule achieves better convergence performance compared to the well-regarded constant stepsize $\alpha = \tfrac{2}{M+m}$,  where $m$ and $M$ denote the strong convexity and gradient-Lipschitz parameters, respectively. 
Our stepsize schedule can be viewed as a generalization of the well-known two-step silver stepsize schedule in~[J. M. Altschuler and P. A. Parrilo, {\sl Journal of the ACM}, 72(2):1-38, 2025] from deterministic setting to stochastic optimization.

\end{abstract}

\section{Introduction}

For optimization problems in the form of
\begin{align*}
    \min_{x \in \mathbb{R}^d} f(x),
\end{align*}
where $f : \R^d\to\R$ is a smooth strongly convex function, if exact gradient information $\nabla f(x)$ is available, gradient descent method serves as a fundamental approach~\cite{nocedal2006numerical}. 
Classical analyses often choose stepsizes that enforce per-iteration descent of the objective. However, recent developments show that this per-iteration descent is not necessary for convergence: larger, structured stepsize schedule can permit temporary increases yet still deliver provable and improved convergence rate. Early investigations based on two-step and three-step schemes revealed that employing moderately larger stepsizes, while temporarily violating monotonic decrease of objective, can nevertheless lead to faster overall convergence~\cite{altschuler2018greed}. 
Inspired by this observation, \cite{altschuler2025acceleration} introduced the so-called silver stepsize for strongly convex problems, a carefully structured sequence that further exploits this long-stepsize phenomenon and achieves a convergence rate of $\mathcal{O}(\kappa^{\log_{\rho}2}\log(1/\epsilon))$, which lies between that of the classical constant-stepsize method $\mathcal{O}(\kappa\log(1/\epsilon))$ and the accelerated rate $\mathcal{O}(\sqrt{\kappa}\log(1/\epsilon))$ attained by the Nesterov's accelerated gradient method~\cite{nesterov2013introductory}. 
Here $\kappa$ denotes the condition number, $\epsilon$ is the error tolerance of iterate distance to optimality, and $\rho = 1+\sqrt{2}$.
Unlike small-stepsize rules that guarantee descent at each iteration, the silver stepsize schedule ensures a net decrease in the objective only over multi-step horizons of length $N = 2^k$. 
This result was later extended to the smooth convex setting~\cite{Altschuler2025}, where the horizon length is $N = 2^k-1$.

Building on this foundation, subsequent works have extended and generalized these deterministic schemes to arbitrary numbers of iterations, nonsmooth problems, and minimax formulations. In particular,~\cite{grimmer2024composing,zhang2024accelerated} develop schedules applicable to arbitrary iteration lengths when the total number of iterations is known in advance, while~\cite{zhang2024anytime} introduces an anytime variant that removes the need for prior horizon knowledge. Similar acceleration phenomena have also been observed beyond smooth unconstrained settings, including the relaxed proximal point method~\cite{wang2024relaxed} and the proximal gradient method~\cite{bok2024accelerating}. Furthermore,~\cite{altschuler2024acceleration} proposed a randomized schedule for separable convex problems, and~\cite{shugart2025negative} analyzed minimax optimization where specially designed negative stepsizes yield convergence of the standard gradient ascent-descent dynamics. Taken together, these studies reveal a unifying principle: carefully designed stepsize rules, although they may not guarantee descent of the objective at every iteration, can yield substantially better convergence results than classical constant-stepsize schemes.

When only stochastic approximations of gradient are accessible, the role of stepsize becomes even more delicate. It not only governs the convergence of iterates but also controls the accumulation of noise within stochastic estimates. An effective schedule must therefore balance two competing goals: maintaining large stepsizes to ensure fast convergence, while mitigating the variance amplification induced by stochastic noise. Motivated by the success of the silver stepsize in the deterministic setting, this naturally leads to the following key question:

\begin{center}
\begin{quote}
{\sl In stochastic optimization, can one design stepsize schedules that generalize the silver stepsize---thus achieving faster convergence and simultaneously controlling variance accumulation?}
\end{quote} 
\end{center}

The design of stepsize schedules in deterministic optimization has been largely driven by the Performance Estimation Problem (PEP) framework, first established in~\cite{drori2014performance}. 
The PEP formulated the exact worst-case performance of a given method over a function class as an infinite-dimensional optimization problem and several tractable relaxations have been proposed to solve it.
Subsequently,~\cite{taylor2017smooth} introduced interpolation conditions that represent the PEP feasible set in finite dimensions, enabling exact semidefinite programming (SDP) formulations. 
By optimizing over a function class, the PEP framework yields tight performance bounds for first-order algorithms~\cite{drori2020efficient, taylor2017exact}. 
However, the dimension of the resulting SDP increases with the number of iterations, making it difficult to obtain analytical closed-form results. 
Consequently, many studies employ computer-assisted approaches to obtain numerical insight and to guide proofs, e.g.,~\cite{dragomir2022optimal,lieder2021convergence,Ryu2020Operator}.
Based on such computer-aided analyzes,~\cite{Grimmer2024Provably} provides theoretical evidence that using larger stepsizes, potentially violating per-iteration descent, can accelerate overall convergence. 

To design stepsize schedules in the stochastic setting, it is essential to establish an appropriate PEP formulation that captures the behavior of stochastic oracles. While most existing works focus on deterministic settings, several attempts have been made to extend PEP to stochastic optimization. One line of research examines the worst-case behavior of inexact methods by analyzing a single realization that corresponds to the worst possible sequence of oracle outcomes~\cite{de2020worst, devolder2014first,taylor2017exact}. In this way, the stochastic setting is effectively treated from a deterministic perspective. Moreover,~\cite{vernimmen2025worst} provided a tight analysis of relatively inexact gradient descent methods.
Another line of research investigates finite-support stochastic settings, in which the stochastic oracle at each point can take at most $n$ possible realizations. In such cases, stochastic methods can still be analyzed exactly through the PEP framework by enumerating all possible realizations of stochastic approximation~\cite{hu2017unified,taylor2019stochastic}. However, the resulting SDP dimension grows not only with the number of iterations but also with the number of possible stochastic realizations per iteration, leading to a total complexity of order $\mathcal{O}(N^n)$ for $N$ iterations and $n$ stochastic supports. This exponential growth imposes substantial computational challenges even for numerical solvers, making it exceedingly difficult to obtain closed-form theoretical results.

\subsection{Contributions}
Building upon the deterministic PEP framework, this paper introduces a tractable stochastic PEP formulation that characterizes the expected performance of stochastic gradient (SG) methods under common assumptions. This formulation further motivates a two-step long stepsize schedule that accelerates the convergence of SG algorithms for minimizing strongly convex and smooth functions. Our main contributions are summarized as follows. 

\textbf{Generalization of the two-step silver stepsize~\cite{altschuler2025acceleration}.} We propose a two-step long stepsize schedule for SG methods applied to $m$-strongly convex and $M$-smooth problems. Our schedule generalizes the well-known two-step silver stepsize originally introduced in the deterministic setting by~\cite{altschuler2025acceleration}. Analogous to the acceleration observed in the deterministic case, we show that our proposed schedule achieves superior expected convergence performance compared to the classical constant stepsize $\frac{2}{M+m}$, especially when the initial optimality gap dominates the variance of gradient estimates.

\textbf{Theoretical guarantees.} We establish a PEP framework for analyzing the two-step worst-case performance of SG methods on $m$-strongly convex and $M$-smooth functions. Under common assumptions, we rigorously prove that our two-step stepsize schedule consistently reduces the expected distance to the optimal solution compared with the constant stepsize $\frac{2}{M+m}$. While the empirical benefits of SG algorithms employing silver stepsize schedules were previously investigated in~\cite{vernimmen2025empirical}, this paper, to the best of our knowledge, presents the first theoretical justification for the accelerated convergence of a two-step long stepsize schedule relative to the constant stepsize $\frac{2}{M+m}$.

\textbf{Numerical validation.}
We conduct numerical experiments to support our theoretical findings and to compare our two-step stepsize schedule against the constant stepsize $\frac{2}{M+m}$. The results confirm the superior performance of our approach, particularly in regimes where the initial optimality gap substantially exceeds the variance of the stochastic gradient estimates---a setting that commonly arises in practical applications.

\subsection{Notations}
Let $\N:=\{1,2,3,\ldots\}$ denote the positive natural number set, $\R$ denote the real number set, and $\N_{>s}$ and $\R_{>s}$ (resp., $\N_{\geq s}$ and $\R_{\geq s}$) denote the set of all positive natural numbers and real numbers strictly greater than (resp., greater than or equal to) $s\in\R$, respectively. For $n\in\N$, let $[n]$ denote the set $\{1,2,\ldots,n\}$, $\R^n$ denote the set including all $n$-dimensional real vectors, and $\mathbf{1}_n$ denote the $n$-dimensional all-ones vector.
For $(d,n)\in\N\times\N$, let $\R^{d\times n}$ denote the set including all $d$-by-$n$-dimensional real matrices, $I_n$ denote the $n$-by-$n$-dimensional identity matrix, and $J_n$ denote the $n$-by-$n$-dimensional all-ones matrix, i.e., $J_n:=\mathbf{1}_n\mathbf{1}_n^\top$. We use $\mathbf{0}$ to denote all-zeros vectors or matrices with dimensions implied by the context. For $x\in\mathbb{R}^n$ and $i\in[n]$, let $[x]_i$ denote the $i$th-element of vector $x$.
For $\{x,y\}\subset\mathbb{R}^n$, let's define $\langle x,y\rangle:=x^\top y$ and $\|x\|:=\sqrt{x^\top x}$. For a symmetric matrix $X\in\R^{n\times n}$, i.e., $X = X^{\top}$, let $X\succeq \mathbf{0}$ (resp., $X \succ \mathbf{0}$) denote the matrix $X$ is positive semi-definite (resp., positive definite). For $k\in\N$ and any two matrices $A\in\R^{d_A\times n_A}$ and $B\in\R^{d_B\times n_B}$, let's define $A\oplus B:= \begin{bmatrix} A & \mathbf{0} \\ \mathbf{0} & B \end{bmatrix} \in \R^{(d_A + d_B)\times(n_A + n_B)}$ and $A^{\oplus k} := \underbrace{A\oplus A \oplus \cdots \oplus A}_{k\textup{ number of A's}}$. For $k\in\N$ and $\{A_i\}_{i=1}^k$, where each $A_i\in\R^{n_i\times n_i}$ is a squared matrix with $n_i\in\N$, let's denote the block diagonal matrix constructed by $\{A_i\}_{i=1}^k$ as $\text{diag}(A_1,\dots,A_k) := A_1\oplus A_2\oplus\cdots\oplus A_k$.
For any $s\in\R$, let $\lfloor s \rfloor$ denote the floor function, representing the largest integer number smaller than or equal to $s$, and let $\lceil s \rceil$ denote the ceiling function, representing the smallest integer number greater than or equal to $s$. For $(d,n)\in\N_{>0}\times\N$, we denote $\text{mod}(n,d)\in\{0\}\cup[d-1]$ as the remainder after $d$ divides $n$, i.e., $\text{mod}(n,d) = n - d\cdot \left\lfloor \frac{n}{d} \right\rfloor$.

\subsection{Organization}

The remainder of the paper is organized as follows.
In Section~\ref{sec.problem_setting}, we introduce the performance estimation framework tailored to the stochastic setting considered in this paper. 
Section~\ref{sec.theory} presents our proposed two-step stepsize schedule for the stochastic gradient method and analyzes its properties and convergence guarantees in comparison with the constant-stepsize scheme.  
Section~\ref{sec.numerical} presents numerical  experiments that validate our theoretical findings and provide further insights. Finally, Section~\ref{sec.conclusion} concludes this paper and outline potential directions for future research.

\section{Problem Setting}\label{sec.problem_setting}

In this paper, we focus on the optimization problem
\begin{equation}\label{eq:main_problem}
    \min_{x \in \R^d} f(x) := \E_{\xi}[F(x,\xi)],
\end{equation}
where $\xi$ is a random variable defined with respect to probability space $(\Xi,\mathscr{F},\PP_{\xi})$ and the objective function $f:\R^d\to\R$ belongs to function class $\cFMM$, denoting the class of proper, closed, and $m$-strongly convex functions with $M$-Lipschitz continuous gradients. Namely, with given constant parameters $\{m,M\}\subset\R_{>0}$, for any $\{x,y\} \subset \R^d$, function $f:\R^d\to\R$ satisfies that
\begin{equation}\label{eq:setting}
    f(y) \ge f(x) + \inner{\nabla f(x)}{y-x} + \frac{m}{2}\norm{y-x}^2 \quad \text{and} \quad
    \norm{\nabla f(x) - \nabla f(y)} \le M \norm{x-y}.
\end{equation}
A widely used method for solving \eqref{eq:main_problem} at scale is the stochastic gradient (SG) method~\cite{Bottou2018Optimization}.
Given an iterate $x\in\R^d$ and random variable $\xi$, SG algorithm first evaluates a stochastic gradient estimate $g(x,\xi)\approx\nabla f(x)$ and then updates $x\in\R^d$ along the $-g(x,\xi)$ direction, i.e., the next iterate is $x-\alpha g(x,\xi)$, where $\alpha\in\R_{>0}$ is step size. Throughout the whole paper, we make the following assumption regarding stochastic gradient estimates $g(x,\xi)$.

\begin{ass}\label{ass:oracle}
At any $x\in\R^d$, stochastic gradient estimate $g(x,\xi)$ always performs as an unbiased estimator of true gradient $\nabla f(x)$ with bounded variance and finite support, i.e., for all $x\in\R^d$, there exists $n\in\N$ such that $\{g(x,\xi^{(i)})\}_{i=1}^n$ includes all possible stochastic gradient estimate realizations at $x\in\R^d$. In addition, for any $x\in\R^d$, each stochastic gradient estimate $g(x,\xi^{(i)})$, where $i\in\{1,2,\ldots,n\}$, is realized with equal probability. As a summary, for all $x\in\R^d$, there exist $n\in\N$ and $\sigma\in\R_{>0}$ such that
\begin{align*}
\frac{1}{n}\sum_{i=1}^n g(x,\xi^{(i)}) = \mathbb{E}_{\xi}[g(x,\xi)] = \nabla f(x) \quad \text{and} \quad \frac{1}{n}\sum_{i=1}^n \|g(x,\xi^{(i)}) - \nabla f(x)\|^2 = \mathbb{E}_{\xi}[\|g(x,\xi) - \nabla f(x)\|^2] \le \sigma^2.
\end{align*}
\end{ass}

We make the following remark regarding Assumption~\ref{ass:oracle}.

\begin{remark}
Assumption~\ref{ass:oracle} is standard in stochastic optimization. Assuming stochastic gradient estimates to be unbiased estimator of the true gradient with bounded variance is a condition commonly adopted in the analysis of stochastic approximation methods~\cite{birge1997introduction, Bottou2018Optimization, ermoliev1988numerical}. 
Moreover, the finite-support assumption with equal sampling probability is naturally satisfied by many settings, including finite-sum problems. 

Furthermore, as recently shown in~\cite{rubbens2025computer}, the PEP formulated under such a finite-support stochastic oracle yields a valid bound for more general stochastic settings that are with infinitely many supports.
This observation implies that the theoretical results developed in this work---notably Theorems~\ref{theo:dual_feasibility} and~\ref{theo:lowerbound}---also remain valid and informative in much broader stochastic frameworks.
\end{remark}

\subsection{Performance Estimation Problem Model}

In this subsection, we introduce a PEP for a two-step SG algorithm under Assumption~\ref{ass:oracle}. The goal is to characterize the worst-case expected progress after two consecutive SG algorithm iterations subject to the structural properties of the problem class. With stepsize schedule $\{\alpha,\beta\}\subset\R_{>0}$, the two-step PEP can be formally written as
\begin{align}
    \sup_{\substack{\{x_0, x_1, x_2, x_*\}\subset\R^d,\\ \{\nabla f(x_0), \nabla f(x_1), \nabla f(x_*)\}\subset\R^d, \\
    \{g(x_0, \xi_0), g(x_1, \xi_1)\}\subset\R^d, \\
    f \in \cFMM}}
    &\quad \mathbb{E}_{\xi_0,\xi_1}[\|x_2 - x_\ast\|^2] \label{pep:one-step} \\
    \text{s.t.} \quad\quad\quad\quad\ \ & \nabla f(x_*) = \mathbf{0} \quad\text{and}\quad \|x_0 - x_\ast\|^2 \le R^2, \label{eq:initial_condition} \\
        & x_{1} = x_0 - \alpha g(x_0, \xi_0) \quad\text{and}\quad x_{2} = x_1 - \beta g(x_1, \xi_1), \label{eq:iterative_update} \\
        & \nabla f(x_k) = \E_{\xi_k}[g(x_k,\xi_k)] = \frac{1}{n}\sum\limits_{i=1}^n g(x_k, \xi_k^{(i)}) \ \  \forall \  k\in\{0, 1\}, \label{eq:expected_val_condition} \\
        & \sigma^2 \geq \E_{\xi_k}[\|g(x_k,\xi_k) - \nabla f(x_k)\|^2]  = \frac{1}{n}\sum\limits_{i=1}^n\|g(x_k, \xi_k^{(i)}) - \nabla f(x_k)\|^2 \ \ \forall \  k\in\{0, 1\}, \label{eq:variance_condition}
\end{align}
where~\eqref{eq:initial_condition} states the optimality condition and requires the distance between initial iterate $x_0$ and optimal solution $x_*$ is no larger than $R\in\R_{>0}$, \eqref{eq:iterative_update} describes the iterative update among $\{x_0,x_1,x_2\}\subset\R^d$, \eqref{eq:expected_val_condition}--\eqref{eq:variance_condition} requires stochastic gradient estimates to be unbiased estimators with bounded variance while satisfying Assumption~\ref{ass:oracle}. 

It's worth mentioning that the constraint $f \in \cFMM$ introduces an infinite-dimensional and generally intractable condition. As shown in~\cite{taylor2017smooth}, such an infinite-dimensional condition can be further replaced by an interpolation condition that only relies on a finite set of iterate information without changing solutions of optimization problem~\eqref{pep:one-step}--\eqref{eq:variance_condition}. In this way, we can reduce the condition $f\in\cFMM$ to a finite-dimensional form. We present such result in the following proposition. 

\begin{proposition}\label{prop:Qij}
(See~\cite[Theorem~4]{taylor2017smooth}.) Let's consider $0 < m < M < +\infty$ and a finite index set $\mathcal{K}$. Then $\{(x_k,\nabla f(x_k),f(x_k))\}_{k\in\mathcal{K}}$ is $\cFMM$-interpolable if and only if for any $i\in\mathcal{K}$ and $j\in\mathcal{K}$, it holds that
\begin{equation}\label{eq:Qij}
\begin{aligned}
    Q_{i, j} := 2(M-m)\big(f(x_i) - f(x_j)\big)
    + 2\langle M \nabla f(x_j) &- m \nabla f(x_i),\, x_j - x_i \rangle \\
    &- \|\nabla f(x_i) - \nabla f(x_j)\|^2
    - M m \|x_i - x_j\|^2 \ge 0.
\end{aligned}
\end{equation}
\end{proposition}

Denoting $\Omega$ as the index set of realizations, i.e., 
\begin{equation}\label{eq:Omega_index_set}
\Omega = \{\,0,\, 1^{(1)},\,\dots,\,1^{(n)},\, * \,\},
\end{equation}
and applying the interpolation result in Proposition~\ref{prop:Qij} to~\eqref{pep:one-step}--\eqref{eq:variance_condition}, we may formulate the finite-dimensional problem
\begin{align}
    \sup_{\substack{\{x_i\}_{i\in\Omega}\subset\R^d, \\ \{\nabla f(x_i)\}_{i\in\Omega}\subset\R^d, \\
    \{g(x_j, \xi_j^{(i)})\}_{j\in\Omega\backslash\{*\}}\subset\R^d \ \forall i\in[n]}} &\quad \frac{1}{n^2}\sum_{i=1}^n\sum_{j=1}^n \left\|x_{1^{(i)}} - g(x_{1^{(i)}}, \xi_{1^{(i)}}^{(j)}) - x_\ast\right\|^2  \nonumber{} \\
    \text{s.t.} \quad\quad\quad\ & \nabla f(x_*) = \mathbf{0} \quad\text{and}\quad \|x_0 - x_\ast\|^2 \le R^2, \nonumber{} \\
        & x_{1^{(i)}} = x_0 - \alpha g(x_0, \xi_0^{(i)}) \quad \forall \ i\in[n], \label{pep:finite_dimension} \\
        & \frac{1}{n}\sum\limits_{i=1}^n g(x_j, \xi_j^{(i)}) = \nabla f(x_j) \quad \forall \ j\in \Omega\backslash\{*\}, \nonumber{} \\
        & \frac{1}{n}\sum\limits_{i=1}^n\|g(x_j, \xi_j^{(i)}) - \nabla f(x_j)\|^2 \leq \sigma^2 \quad \forall \  j\in \Omega\backslash\{*\}, \nonumber{} \\
        & Q_{i, j}\ge 0, \quad \forall\, (i,j) \in \Omega\times\Omega, \label{eq:Qij_constraint}
\end{align}
where $Q_{i,j}$ is defined in~\eqref{eq:Qij} and $\{\xi_k^{(i)}\}_{i=1}^n$ are realizations of random variable $\xi_k$ for $k\in\{0,1\}$ (see Assumption~\ref{ass:oracle}). We highlight that~\eqref{eq:Qij_constraint} comes from the interpolation condition for $m$-strongly convex and $M$-smooth functions (see Proposition~\ref{prop:Qij}), ensuring that there exists some function $\hat{f}\in\cFMM$ such that for any $i\in\Omega$, $(\hat{f}(x_i),\nabla \hat{f}(x_i)) = (f(x_i),\nabla f(x_i))$ in problem~\eqref{pep:finite_dimension}. 
Under Assumption~\ref{ass:oracle} and interpolation conditions, the finite-dimensional reformulation~\eqref{pep:finite_dimension} is equivalent to the original problem~\eqref{pep:one-step}: the expectation terms are replaced by empirical averaging over stochastic realizations, while unbiased and bounded variance requirements are enforced via linear and quadratic relations, respectively.

To solve the finite-dimensional formulation of the two-step PEP, following \cite{taylor2017smooth}, we reformulate it as an SDP by introducing a Gram matrix $H=P^TP$, where
\begin{equation}\label{eq:P_matrix}
    P = \begin{bmatrix}x_0 & g(x_0, \xi_0^{(1)}) & \cdots & g(x_0, \xi_0^{(n)}) & g(x_{1^{(1)}}, \xi_{1^{(1)}}^{(1)}) & \cdots & g(x_{1^{(n)}}, \xi_{1^{(n)}}^{(n)}) \end{bmatrix} \in\R^{d\times(1+n+n^2)}
\end{equation}
captures the information of initial iterate $x_0\in\R^d$ and all possible stochastic gradient estimates. 
Without the loss of generality, let's consider $x_* = \mathbf{0}$, whose optimality condition is $\nabla f(x_*) = \mathbf{0}$. Next, we introduce the following symbolic construction to parameterize the SDP. Let $e_i$ denote the $i$-th column of identity matrix $I_{(1+n+n^2)}$. Defining symbolic vectors
\begin{align}\label{eq.sdp.symobvec}
\bar{x}_0 = e_1, \quad \bar{g}_0^{(i)} = e_{1+i}, \quad\text{and}\quad \bar{g}_{1^{(i)}}^{(j)} = e_{1+n+(n-1)i+j}\ \ \text{for any }(i,j)\in[n]\times[n],
\end{align}
we have
\begin{align*}
x_0 = P \bar{x}_0, \quad g(x_0, \xi_0^{(i)}) = P \bar{g}_0^{(i)}, \quad\text{and}\quad g(x_{1^{(i)}}, \xi_{1^{(i)}}^{(j)}) = P \bar{g}_{1^{(i)}}^{(j)} \ \ \text{for any }(i,j)\in[n]\times[n],
\end{align*}
where $P\in\R^{d\times(1+n+n^2)}$ is defined in~\eqref{eq:P_matrix}. Meanwhile, from SG algorithm and Assumption~\ref{ass:oracle}, for any $(i,j)\in[n]\times[n]$, we further define
\begin{equation}\label{eq.sdp.symobvec_1}
\begin{aligned}
\bar{x}_{1^{(i)}} = \bar{x}_0 - \alpha \bar{g}_0^{(i)}, \quad  \bar{x}_2^{(i,j)} &= \bar{x}_{1^{(i)}} - \beta\bar{g}_{1^{(i)}}^{(j)}, \quad \bar{x}_* = \mathbf{0}, \\
\nabla \bar{f}_0 = \frac{1}{n} \sum_{t=1}^n \bar{g}_0^{(t)}, \quad \nabla \bar{f}_{1^{(i)}} &= \frac{1}{n} \sum_{t=1}^n \bar{g}_{1^{(i)}}^{(t)}, \quad \text{and} \quad \nabla \bar{f}_* = \mathbf{0}
\end{aligned}
\end{equation}
so that
\begin{align*}
x_{1^{(i)}} = P \bar{x}_{1^{(i)}}, \quad  x_2^{(i,j)} &= P \bar{x}_2^{(i,j)}, \quad x_* = P\bar{x}_*, \\
\nabla f(x_0) = P \nabla \bar{f}_0, \quad  \nabla f(x_{1^{(i)}}) = P \nabla \bar{f}_{1^{(i)}} \ \  &\text{for any }(i,j)\in[n]\times[n], \quad \text{and} \quad \nabla f(x_*) = P\nabla\bar{f}_*.
\end{align*}
Next, we define several symbolic matrices corresponding to the objective and constraints of the SDP formulation. First, matrices $\bar{A}_0^{\mathrm{var}}\in\R^{(1+n+n^2)\times(1+n+n^2)}$ and $\bar{A}_{1^{(i)}}^{\mathrm{var}}\in\R^{(1+n+n^2)\times(1+n+n^2)}$, where $i\in[n]$, encode the bounded-variance conditions in Assumption~\ref{ass:oracle}:
\begin{equation}\label{eq:Abar}
\begin{aligned}
&\bar{A}_0^{\text{var}} = \frac{1}{n} \sum_{i=1}^n \big(\bar{g}_0^{(i)} - \nabla \bar{f}_0\big)\big(\bar{g}_0^{(i)} - \nabla \bar{f}_0\big)^{\top} \\
\text{and} \quad &\bar{A}_i^{\text{var}} = \frac{1}{n} \sum_{j=1}^n \big(\bar{g}_{1^{(i)}}^{(j)} - \nabla \bar{f}_{1^{(i)}}\big)\big(\bar{g}_{1^{(i)}}^{(j)} - \nabla \bar{f}_{1^{(i)}}\big)^{\top} \quad \text{for any } i\in[n].
\end{aligned}
\end{equation}
For any $(i,j)\in\Omega\times\Omega$, we further use matrix $\bar{B}_{i,j}\in\R^{(1+n+n^2)\times(1+n+n^2)}$ to encode the $\cFMM$ interpolation conditions in Proposition~\ref{prop:Qij}, such as
\begin{equation}\label{eq:Bbar}
    \bar{B}_{i,j} = 
    \begin{bmatrix}\bar{x}_i & \bar{x}_j & \nabla \bar{f}_i & \nabla \bar{f}_j\end{bmatrix}
    \begin{bmatrix}
        -Mm & Mm & m & -M \\
        Mm & -Mm & -m & M \\
        m & -m & -1 & 1 \\
        -M & M & 1 & -1
    \end{bmatrix}
    \begin{bmatrix} \bar{x}_i & \bar{x}_j & \nabla \bar{f}_i & \nabla \bar{f}_j \end{bmatrix}^{\top}.
\end{equation}
Meanwhile, let's consider another matrix $\bar{C}\in\R^{(1+n+n^2)\times(1+n+n^2)}$ that encodes the distance to optimal solution $x_*=\mathbf{0}$ after two steps of the SG algorithm with stepsize schedule $(\alpha,\beta)$, where
\begin{equation}\label{eq:Cbar}
    \bar{C} = \frac{1}{n^2} \sum_{i=1}^n \sum_{j=1}^n \big( \bar{x}_0 - \alpha \bar{g}_0^{(i)} - \beta \bar{g}_{1^{(i)}}^{(j)} \big) \big( \bar{x}_0 - \alpha \bar{g}_0^{(i)} - \beta \bar{g}_{1^{(i)}}^{(j)} \big)^{\top}.
\end{equation}
Finally, we apply matrix
\begin{equation}\label{eq:AR}
    \bar{A}_R = \bar{x}_0 \bar{x}_0^{\top} \in\R^{(1+n+n^2)\times(1+n+n^2)}
\end{equation}
to capture the initial distance constraint (i.e., $\|x_0 - x_*\|^2\leq R^2$) in problem~\eqref{pep:finite_dimension}.

With the help of~\eqref{eq:Abar}--\eqref{eq:AR}, the finite-dimensional PEP problem~\eqref{pep:finite_dimension} can be formulated as the following SDP problem, whose objective and constraints are presented by the corresponding symbolic matrix: 
\begin{align}\label{pep:primal_prob}
    \sup_{\substack{ H \succeq \mathbf{0},\\ P \in \mathbb{R}^{d \times (1+n+n^2)} \\
    f_i \in \mathbb{R},\ \forall i \in \Omega}} \quad & \mathrm{Tr}(\bar{C} H) \\
    \text{s.t.} \ \quad\quad & 2(M - m)(f_i - f_j) + \mathrm{Tr}(\bar{B}_{i,j} H) \ge 0 \quad \forall (i,j) \in \Omega \times \Omega \text{ with }i\neq j, \nonumber\\
    & \mathrm{Tr}(\bar{A}_R H) \le R^2, \nonumber\\
    & \mathrm{Tr}(\bar{A}_i^{\text{var}} H) \le \sigma^2 \quad \forall\ i\in\{0, 1, ..., n\},    \label{eq.bounded_variance_constraint}\\
    & H = P^TP, \label{eq.sdpcons}
\end{align}
where $\Omega$ is the index set defined in~\eqref{eq:Omega_index_set}, $R\in\R_{>0}$ is an upper bound of the distance between initial iterate $x_0\in\R^d$ and $x_* = \mathbf{0}$, and $\sigma\in\R_{>0}$ serves as an upper bound of the standard deviation of stochastic gradient estimates; see Assumption~\ref{ass:oracle}. Following the procedure in~\cite{taylor2017smooth}, we may further derive the dual formulation of problem~\eqref{pep:primal_prob} without variable $P\in\mathbb{R}^{d \times (1+n+n^2)}$ and constraint~\eqref{eq.sdpcons}, which is
\begin{align}
    \inf_{\substack{\tau\in\R_{\geq 0},\\  \mu_i\in\R_{\geq 0} \text{ for all }i\in\{0\}\cup[n], \text{ and } \\ \lambda_{i,j}\in\R_{\geq 0} \text{ for all }(i,j)\in\Omega\times\Omega\text{ with }i\neq j}} \quad & \tau R^2 + \sum_{i=0}^n \mu_i \sigma^2 \label{prob.dual.sdp} \\
    \text{s.t.} \quad\quad\quad\quad\quad\ \  
        & \tau \bar{A}_R + \sum\limits_{i=0}^n\mu_i \bar{A}_i^{\text{var}} - \sum\limits_{(i, j)\in\Omega\times\Omega \text{ with }i\neq j}\lambda_{i, j}\bar{B}_{i, j} -\bar{C} \succeq \mathbf{0}, \label{eq.dual.sdpconstraint}\\
        & \sum_{i\in\Omega\backslash\{j\}} \lambda_{ij} = \sum_{i\in\Omega\backslash\{j\}} \lambda_{ji} \quad \forall j\in\Omega.\nonumber
\end{align}
We conclude this section with a final remark regarding the relation between~\eqref{pep:primal_prob} and~\eqref{prob.dual.sdp}.
\begin{remark}\label{rem:prime-dual-equal}
When $d > (1 + n + n^2)$, 
the primal formulation~\eqref{pep:primal_prob} is equivalent to the variant obtained by omitting the constraint~\eqref{eq.sdpcons}, as demonstrated in~\cite{taylor2017smooth}. In this case, the feasible regions characterized by $H \succeq \mathbf{0}$ and by $H = P^T P \succeq \mathbf{0}$ coincide. Hence, the formulation~\eqref{prob.dual.sdp} constitutes the exact dual of~\eqref{pep:primal_prob} and provides a valid upper bound on the primal objective.  

On the other hand, when $d \le (1 + n + n^2)$, which is the case that most commonly occurs, the relaxation obtained by removing~\eqref{eq.sdpcons} yields a problem whose objective value upper bounds that of~\eqref{pep:primal_prob}. 
Consequently, its dual problem~\eqref{prob.dual.sdp} continues to provide a valid upper bound for~\eqref{pep:primal_prob}.
\end{remark}

\section{Stepsize Schedule}\label{sec.theory}
In this section, we aim to introduce our proposed two-step stepsize schedule, and present its theoretical properties. In the rest of this paper, without the loss of generality, we always consider $m=1$ and $M > 1$. In addition, we introduce another parameter $v\in\left[0, \frac{(M-1)n}{(n-1)M}\right]$\footnote{By convention, $\frac{(M-1)n}{(n-1)M} = +\infty$ when $n=1$.} that plays a crucial role in our stepsize schedule selection.  
In the deterministic setting, the classical two-step silver stepsize schedule~\cite{altschuler2025acceleration} is obtained directly from the dual formulation~\eqref{prob.dual.sdp}. 
In this formulation, the positive semi-definite (PSD) constraint~\eqref{eq.dual.sdpconstraint} involves a symmetric matrix whose entries depend on the stepsizes and the dual variables. 
By setting every entry of this matrix to zero, one obtains an all-zeros matrix that is trivially PSD for any two-step stepsize schedule selection. 
The resulting algebraic system, obtained by equating all entries to zero, simultaneously determines the dual variables as well as the two-step silver stepsize schedule, which has been proven optimal in the deterministic case.

However, in the stochastic setting, the all-zeros matrix used in the deterministic construction fails to yield a meaningful stepsize schedule, degenerating to the trivial constant stepsize $\tfrac{2}{M+m}$.
This limitation arises because the size of the associated SDP increases with the number of stochastic supports $n$: while the deterministic two-step case involves a $3\times3$ matrix, the stochastic formulation requires a matrix of dimension $(1+n+n^2)\times(1+n+n^2)$. 
From the perspective of generalizing the silver stepsize to the stochastic setting, we introduce a structured parametric matrix pattern, which depends on the selection of $v$ and is not necessarily $\mathbf{0}$, in place of the all-zeros matrix. 
Given Lipschitz gradient constant $M > 1$ and the number of stochastic supports $n$, each entry of the our parametric matrix is fully determined once the parameter $v$ is fixed, thereby defining a structured, non-trivial matrix that reflects stochastic variability. 
In the same manner as in the deterministic derivation, equating our parametric matrix with the left-hand-side matrix in the PSD constraint~\eqref{eq.dual.sdpconstraint} of the dual formulation yields an algebraic system whose solution determines the dual variables and the two-step stepsize schedule for the stochastic setting.

Our parametric matrix naturally recovers the all-zeros matrix achived by the classical two-step silver stepsize schedule in the deterministic regime. 
When $n = 1$, our matrix reduces to an all-zeros $3\times3$ matrix, reproducing exactly the deterministic construction. 
In the stochastic case with $n > 1$ and $\sigma > 0$, setting $v = 0$ yields an equation system identical to that of the deterministic formulation, except for additional dual variables associated with the bounded-variance constraint~\eqref{eq.bounded_variance_constraint}, which further results in the classical two-step silver stepsize. 
This setting ($n>1$, $\sigma>0$, and $v=0$) fundamentally mimics scenarios when the initial optimality gap dominates the variance bound of stochastic gradient estimates, e.g., $R \gg \sigma$, where the difference among stochastic gradient estimates is negligible.
In fact, the choice of $v$ is delicate, as it directly reflects the problem's relative stochasticity, quantified by $\frac{\sigma}{R}$ in~\eqref{prob.dual.sdp}. A more detailed discussion is deferred to Remarks~\ref{rem:dual_feasibility} and~\ref{rem:lowerbound}.

The rest of this section focuses on analyzing theoretical properties of our proposed two-step stepsize schedule for minimizing $1$-strongly convex and $M$-smooth functions under Assumption~\ref{ass:oracle}. 
We first establish in~\eqref{eq.stepsize.range}--\eqref{eq.stepsize.beta} the precise conditions that define our proposed two-step stepsize schedule, whose uniqueness is shown in Theorem~\ref{lem:root-interval}.  
Subsequently, by identifying a set of dual multipliers that are feasible to~\eqref{prob.dual.sdp}, Theorem~\ref{theo:dual_feasibility} characterizes the expected error bound of SG methods with our proposed schedule (see~\eqref{eq.upperbound}). 
Eventually, it is shown in Theorem~\ref{theo:lowerbound} that our proposed schedule achieves strictly better performance compared with the popular classical constant stepsize $\frac{2}{M+m}$.

In particular, given arbitrary parameters $(M,n,v)\in\R_{>1}\times\N\times\left[0, \tfrac{(M-1)n}{(n-1)M}\right]$, we will show in Theorem~\ref{lem:root-interval} that there exists a unique two-step stepsize schedule $(\alpha,\beta)\in\R^2$ satisfying   
\begin{align}\label{eq.stepsize.range}
&(\alpha,\beta) \in \left[\tfrac{1}{M}, 1\right) \times \left(\tfrac{\alpha}{M}, \tfrac{M+1}{2M}\right), \\
0 = \ &(M - 1)Mn\alpha^3 - (M+1)(M-2)n\alpha^2 - (4n + (M-1)(n-1)v)\alpha + 2n, \label{eq.stepsize.alpha} \\
\text{and}\ \ 0 = \ &M n \left((-1 + M)^2 n + (1 + M)^2 (-1 + n) v \right) 
\left( (-1+M)\alpha + (1+M) \right)\beta^2 \label{eq.stepsize.beta} \\
&+ 2M \Bigl[
       2(-1+M)Mn(-n + (n-1)v) \alpha^2 \nonumber{} \\ 
&\quad\quad\quad\  
    + (1+M)(-n + (n-1)v)\,((3 - M) n + (-1 + M) (-1 + n) v) \alpha \nonumber{} \\
&\quad\quad\quad\  
    + n\left( n - (M - 4) M n - (1 + M) (3 + M) (-1 + n) v \right) 
   \Bigr]\beta \nonumber{} \\
&+ \Bigl[
     -2 M (-1 + M^2) n (-n + (n - 1) v)\alpha^2 \nonumber{} \\ 
&\quad\ \ 
   + \left( 
        2 (2 + M + 2 M^2 - M^3) n^2 + (-3 + M) (1 + M)^2 (-1 + n) n v + (1 - 
    M) (1 + M)^2 (-1 + n)^2 v^2  
     \right)\alpha \nonumber{} \\
&\quad\ \ 
   + 2(1+M)^2 n(-n + (n - 1)v)
   \Bigr], \nonumber{}
\end{align}
which is denoted by $(\alpha^*,\beta^*)$.
\begin{theorem}\label{lem:root-interval}
For any parameters $(M,n,v)\in\R_{>1}\times\N\times\left[0, \tfrac{(M-1)n}{(n-1)M}\right]$, there exists a unique $\alpha^*\in\left[\frac{1}{M},1\right)$ such that $\alpha=\alpha^*$ is a solution of~\eqref{eq.stepsize.alpha}. Moreover, with the aforementioned $\alpha^*$, there also exists a unique $\beta^*\in\left(\tfrac{\alpha^*}{M}, \tfrac{M + 1}{2M}\right)$ making $(\alpha,\beta) = (\alpha^*,\beta^*)$ satisfy~\eqref{eq.stepsize.beta}.
\end{theorem}
\begin{proof}
See Appendix~\ref{sec.app_proof_1}.
\end{proof}

Notice that whenever $(n-1)v = 0$, a direct computation of our stepsize policy in~\eqref{eq.stepsize.range}--\eqref{eq.stepsize.beta} implies
\begin{align}\label{eq.alignsilver}
\alpha^* = \frac{2}{1 + \sqrt{\,1 - 2M + 2M^2\,}} 
\qquad \text{and} \qquad
\beta^* = \frac{2}{1 + 2M - \sqrt{\,2M^2 - 2M + 1\,}},
\end{align}
which coincide with the classical two-step silver stepsize schedule in the deterministic setting~\cite{altschuler2025acceleration}.\footnote{The computation can also be verified symbolically in \textsc{Mathematica}; see \url{https://github.com/LuweiYY/LongStepSize_SGD}.}
When $n \in \mathbb{N}_{>1}$, the accuracy of stochastic approximations depends on the variance bound $\sigma$, and different choices of the parameter $v$ result in different $(\alpha^*, \beta^*)$. 
Figure~\ref{fig:step-vs-v} illustrates how the values of $(\alpha^*,\beta^*)$ vary with $v$: as $v$ increases, both $\alpha^*$ and $\beta^*$ get smaller. 
This relation indicates that adjusting $v$ directly influences our proposed two-step stepsize schedule. Notably, by selecting a larger $v$, one can obtain smaller stepsizes $(\alpha^*, \beta^*)$, which intuitively matches the expectation of better performance under higher-noise conditions.

\begin{figure}[htbp]
    \centering
    \begin{subfigure}[b]{0.48\textwidth}
        \centering
        \includegraphics[width=\textwidth]{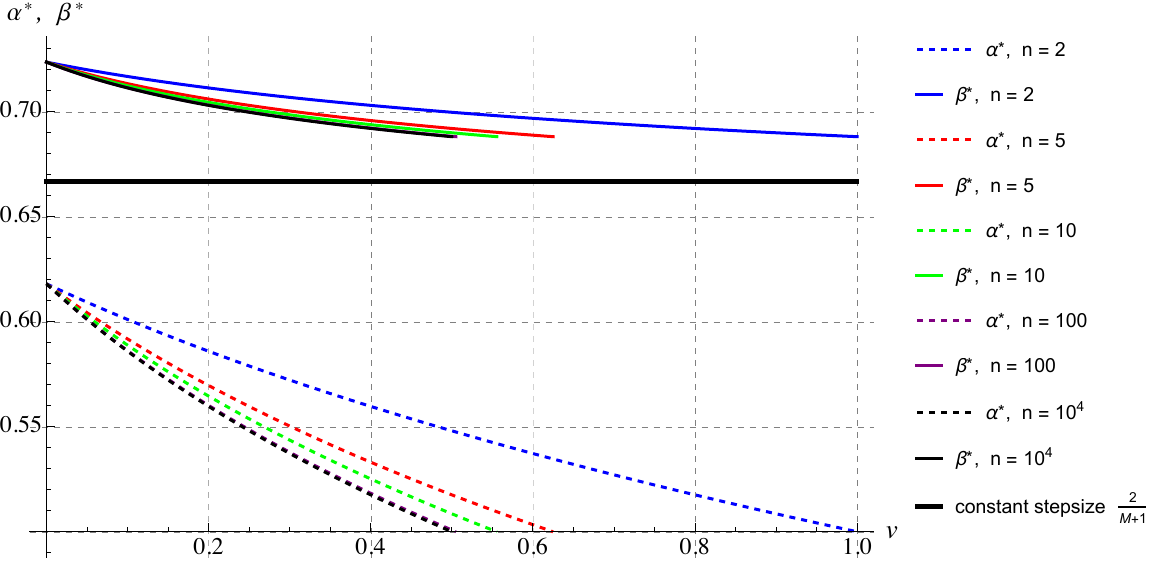}
        \caption{$M=2$}
    \end{subfigure}
    \hfill
    \begin{subfigure}[b]{0.48\textwidth}
        \centering
        \includegraphics[width=\textwidth]{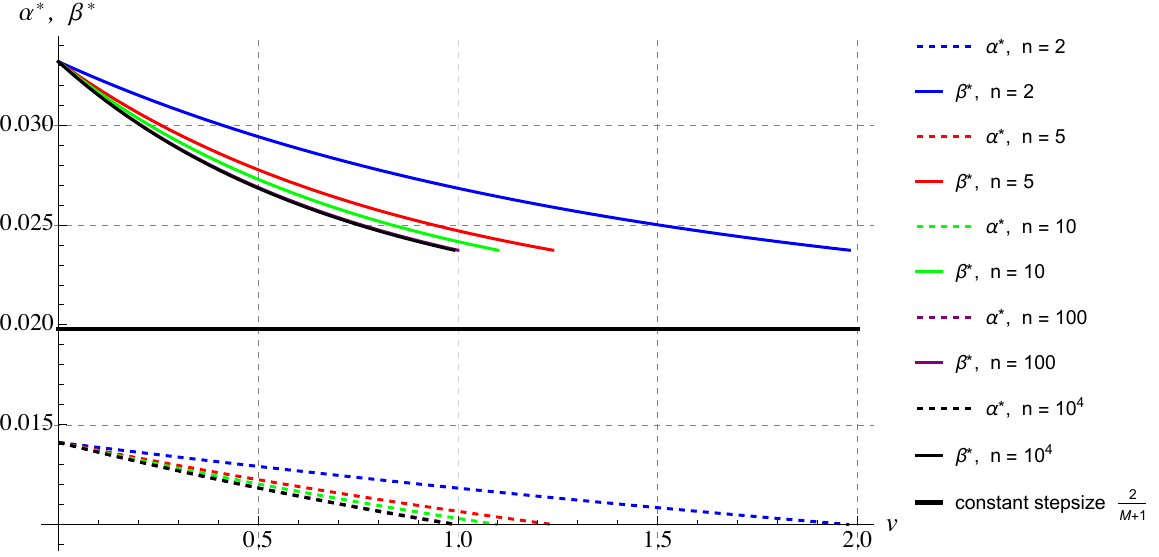}
        \caption{$M=100$}
    \end{subfigure}
    \caption{Two-step stepsize schedule $(\alpha^*, \beta^*)$ versus the parameter $v$ 
    for representative values of $M$ and $n$.}
    \label{fig:step-vs-v}
\end{figure}

Given $(\alpha^*, \beta^*)$, one can find corresponding dual variables that are feasible to the problem~\eqref{prob.dual.sdp}. 
The objective value of~\eqref{prob.dual.sdp} achieved by these feasible dual variables provides an upper bound for the primal objective~\eqref{pep:one-step}. Consequently, by Theorem~\ref{lem:root-interval}, each selection of $v\in\left[0, \tfrac{(M-1)n}{(n-1)M}\right]$ determines a unique two-step stepsize schedule $(\alpha^*, \beta^*)$ and, in turn, a distinct upper bound on the expected error in~\eqref{pep:one-step}. 
This relation is formally established in the following theorem.

\begin{theorem} \label{theo:dual_feasibility}
Suppose $(M,n,v)\in\R_{>1}\times\N\times\left[0, \tfrac{(M-1)n}{(n-1)M}\right]$ are pre-determined parameters and $(\alpha^*,\beta^*)$ is the stepsize schedule described in Theorem~\ref{lem:root-interval}. Based on the index set $\Omega$ defined in~\eqref{eq:Omega_index_set}, we consider variables of the dual formulation~\eqref{prob.dual.sdp} taking values as 
\begin{align}
\lambda_{i,*} &= \frac{\alpha^* + \beta^* - \alpha^*\beta^*}{n(1+M-\alpha^*+M\alpha^*)} \quad\forall i\in\Omega\backslash\{0,*\}, \label{eq.lambda_i*} \\
\lambda_{*,i} &= \frac{M\beta^* + \alpha^*(-1 + M\beta^*)}{Mn(1+M-\alpha^*+M\alpha^*)} \quad\forall i\in\Omega\backslash\{0,*\}, \label{eq.lambda_*i} \\
\lambda_{0,*} &= 0, \label{eq.lambda_0*} \\
\lambda_{*,0} &= n(\lambda_{1^{(1)},*} - \lambda_{*,1^{(1)}}), \label{eq.lambda_*0} \\
\lambda_{0,i} &= \frac{n^2(\lambda_{1^{(1)},*}+M\lambda_{*,1^{(1)}})^2 - n(\lambda_{1^{(1)},*}+\lambda_{*,1^{(1)}})}{2(n - nv + v)} + \frac{\lambda_{1^{(1)},*} - \lambda_{*,1^{(1)}}}{2} \quad \forall i\in\Omega\backslash\{0,*\}, \label{eq.lambda_0i} \\
\lambda_{i,0} &= \frac{n^2(\lambda_{1^{(1)},*}+M\lambda_{*,1^{(1)}})^2 - n(\lambda_{1^{(1)},*}+\lambda_{*,1^{(1)}})}{2(n - nv + v)} - \frac{\lambda_{1^{(1)},*} - \lambda_{*,1^{(1)}}}{2} \quad \forall i\in\Omega\backslash\{0,*\}, \label{eq.lambda_i0} \\
\lambda_{i,j} &= 0 \quad \forall (i,j)\in(\Omega\backslash\{0,*\}) \times (\Omega\backslash\{0,*\}) \text{ with }i\neq j, \label{eq.lambda_ij} \\
\tau &= 1 - M\left(\lambda_{0,*} + \lambda_{*,0} + \sum_{i=1}^n\lambda_{1^{(i)},*} + \sum_{i=1}^n\lambda_{*,1^{(i)}}\right) \label{eq.tau} \\
\mu_0 &= \begin{cases} 0 & \text{if }n=1 \\
\max\Big\{0, \frac{n^2 - (n-1)^2v^2}{(n-1)v}(\lambda_{0,1^{(1)}} + \lambda_{1^{(1)},0}) + n(\lambda_{0,1^{(1)}} + \lambda_{1^{(1)},0}) \\
\quad\quad\quad\quad\quad\quad\quad\quad\quad\quad\quad - 2\alpha^* n(\lambda_{0,1^{(1)}} + M\lambda_{1^{(1)},0}) + (\lambda_{0,*} + \lambda_{*,0})\Big\} &\text{if }n\in\N_{>1}, \end{cases} \label{eq.mu_0} \\
\text{and} \quad \mu_i &= \begin{cases} 0 & \text{if }n=1 \\ \max\left\{0, \left(1 - \frac{(n-1)v}{n}\right)(\lambda_{0,1^{(1)}} + \lambda_{1^{(1)},0}) + (\lambda_{*,1^{(1)}} + \lambda_{1^{(1)},*})\right\} &\text{if }n\in\N_{>1}
\end{cases} \quad \forall i\in[n]. \label{eq.mu_i}
\end{align}
Then the following statements hold true.
\begin{enumerate}[label=(\roman*)]
    \item All variables of the dual formulation~\eqref{prob.dual.sdp} are non-negative.
    \item The positive semi-definiteness constraint and the equality constraint in~\eqref{prob.dual.sdp} are both satisfied.
    \item With the help of $\tau$ and $\{\mu_i\}_{i=0}^n$, there exists a valid upper bound of the optimal objective value of problem~\eqref{pep:one-step}, i.e.,
    \begin{align}\label{eq.upperbound}
        \mathbb{E}_{\xi_0,\xi_1}[\|x_2 - x_\ast\|^2] \leq \tau R^2 + \sum_{i=0}^n\mu_i\sigma^2.        
    \end{align}
\end{enumerate}
\end{theorem}
\begin{proof}
See Appendix~\ref{sec.app_proof_2}.
\end{proof}

\begin{remark}\label{rem:dual_feasibility}
After fixing problem dependent parameters $(M,n)\in\mathbb{R}_{>1}\times\mathbb{N}$, Theorem~\ref{theo:dual_feasibility} provides a feasible solution to the dual PEP problem~\eqref{prob.dual.sdp}, parametrized by $v \in \left[0, \tfrac{(M-1)n}{(n-1)M}\right]$ and the corresponding our proposed two-step stepsize schedule $(\alpha^*(v), \beta^*(v))$. 
This dual feasible solution induces an upper bound on the optimal objective value of primal problem~\eqref{pep:one-step}; see~\eqref{eq.upperbound}. 
Because $(\alpha^*(v), \beta^*(v))$ depends on $v$, varying $v$ produces different dual solutions and, consequently, different right-hand-side value of~\eqref{eq.upperbound}. 
As $(R,\sigma)\subset\R_{>0}$ are pre-determined parameters, the influence of $v$ on the right hand side of~\eqref{eq.upperbound} can be studied via
\begin{align}\label{eq.objfunc}
    h(v, \tfrac{\sigma}{R}) := \tau(v) + \mu(v)\left(\frac{\sigma}{R}\right)^2\quad \text{with} \quad
    \mu(v) := \sum_{i=0}^n \mu_i(v),
\end{align}
where $\tau(v)$, $\mu_0(v)$, and $\{\mu_i(v)\}_{i\in[n]}$ directly result from~\eqref{eq.tau}--\eqref{eq.mu_i} by varying $v\in\left[0, \tfrac{(M-1)n}{(n-1)M}\right]$.
Accordingly, $R^2 h(v, \tfrac{\sigma}{R})$ represents the right hand side of~\eqref{eq.upperbound}, an upper bound of $\mathbb{E}_{\xi_0,\xi_1}[\|x_2 - x_\ast\|^2]$.
When $v = 0$, as shown in~\eqref{eq.alignsilver}, the stepsize schedule recovers the deterministic two-step silver stepsize in~\cite{altschuler2025acceleration}, and~\eqref{eq.tau} indicates
\begin{align}\label{eq.tauSilver}
    \tau(0) = \left(\frac{\sqrt{2M^2 - 2M + 1} - M}{2 + \sqrt{2M^2 - 2M + 1} - M}\right)^2,
\end{align}
which also matches the convergence rate of the deterministic two-step silver stepsize schedule in~\cite{altschuler2025acceleration} \footnote{This computation can be verified symbolically via \textsc{Mathematica}; see \url{https://github.com/LuweiYY/LongStepSize_SGD}.}.
However, direct computation of~\eqref{eq.mu_0}, \eqref{eq.mu_i}, and~\eqref{eq.objfunc} yields $\mu(0) = +\infty$, implying that although $\tau(0)$ recovers the deterministic scenario, the variance-related term $\mu(v)\frac{\sigma^2}{R^2}$ diverges, leading to an unbounded $h(v, \tfrac{\sigma}{R})$ in~\eqref{eq.objfunc} within our stochastic regime.

As illustrated in Figures~\ref{fig:mu-v}--\ref{fig:tau-v}, increasing $v$ makes $\tau(v)$ increase but $\mu(v)$ decrease sharply near $v=0$. 
Hence, the parameter $v$ directly controls the trade-off between $\tau(v)$ and $\mu(v)$. 
For any fixed $\frac{\sigma}{R}$, the optimal $v$ that minimizes $h(v,\tfrac{\sigma}{R})$ in~\eqref{eq.objfunc} should balance these two terms. 
When $\frac{\sigma}{R}$ decreases towards zero, a sufficiently small value of $v \approx 0$ would be preferred. On the other hand, as $\frac{\sigma}{R}$ increases, a large value of $v$ is needed to compensate for the variance through a small $\mu(v)$ term. This insight has also been validated numerically; see Figure~\ref{fig:obj-vs-v}. As shown in Figure~\ref{fig:obj-vs-v}, for each fixed $(R, \sigma)$ parameters, there exists an optimal choice of $v$ that minimizes $h(v,\tfrac{\sigma}{R})$ in~\eqref{eq.objfunc}.

\begin{figure}[htbp]
    \centering
    \begin{subfigure}[t]{0.32\textwidth}
        \centering
        \includegraphics[width=\textwidth]{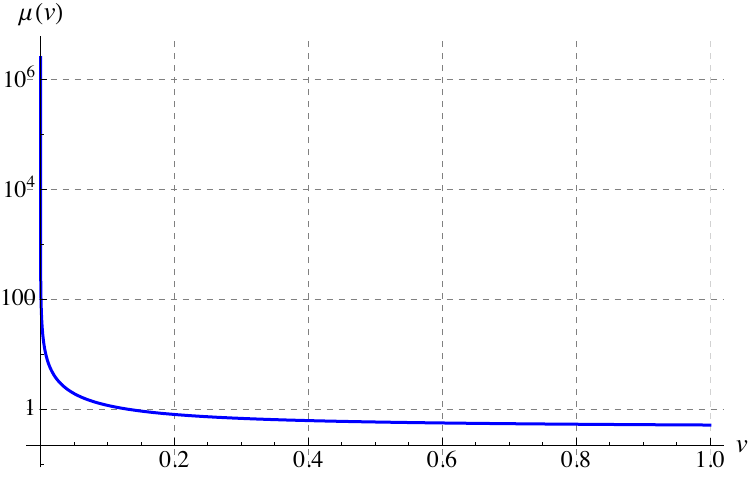}
        \caption{Plot of $\mu(v)$ versus $v$.}
        \label{fig:mu-v}
    \end{subfigure}
    \hfill
    \begin{subfigure}[t]{0.32\textwidth}
        \centering
        \includegraphics[width=\textwidth]{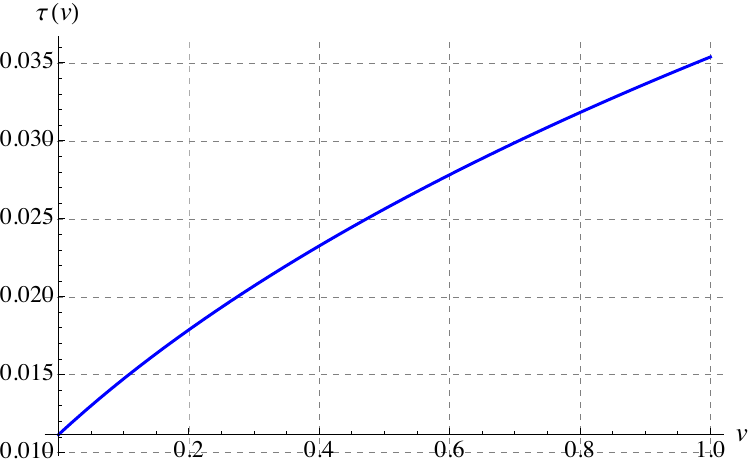}
        \caption{Plot of $\tau(v)$ versus $v$.}
        \label{fig:tau-v}
    \end{subfigure}
    \hfill
    \begin{subfigure}[t]{0.32\textwidth}
        \centering
        \includegraphics[width=\textwidth]{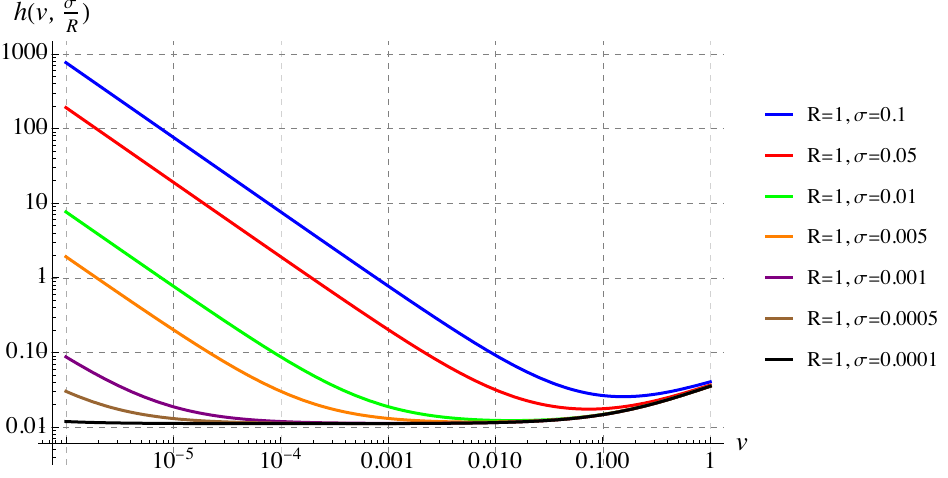}
        \caption{Plot of $h(v,\tfrac{\sigma}{R})$ versus $v$.}
        \label{fig:obj-vs-v}
    \end{subfigure}

    \caption{Dependence of $\mu(v)$, $\tau(v)$, and $h(v,\tfrac{\sigma}{R})$ on $v$ when $(M,n) = (2,2)$.}
    \label{fig:v-analysis}
\end{figure}
\end{remark}

Our final theorem demonstrates that, under the common condition of $\frac{\sigma}{R}$ being small enough, our proposed two-step stepsize schedule (presented in~\eqref{eq.stepsize.range}--\eqref{eq.stepsize.beta}) strictly enhances the convergence performance of an SG algorithm compared with the classical constant stepsize $\frac{2}{M+1}$ ($=\frac{2}{M+m}$ in the setting of $m=1$).

\begin{theorem}\label{theo:lowerbound}
Let $(M, n, R, \sigma) \in \R_{>1} \times \N \times \R_{>0} \times \R_{>0}$ be pre-determined fixed parameters. Following Theorem~\ref{lem:root-interval}, we know that each $v \in \left[0, \tfrac{(M-1)n}{(n-1)M}\right]$ determines a unique two-step stepsize schedule $(\alpha^*, \beta^*)$. Let's denote $h_{\mathrm{silver}}(v, R, \sigma)$ and $h_{\mathrm{constant}}(R, \sigma)$ as optimal objective values of problem~\eqref{pep:one-step} under our proposed stepsize schedule $(\alpha^*, \beta^*)$ and constant stepsize schedule $\alpha=\beta=\tfrac{2}{M+1}$, respectively. Then there exists $\mathscr{U}(M,n) > 0$ such that if $\frac{\sigma}{R} < \sqrt{\mathscr{U}(M,n)}$, then
we can always find some $\bar v \in \left[0, \tfrac{(M-1)n}{(n-1)M}\right]$ satisfying
\begin{align}\label{eq.performance.ratio}
    \frac{h_{\mathrm{silver}}(\bar v, R, \sigma)}{h_{\mathrm{constant}}(R, \sigma)}
    \le \mathscr{C} < 1,
\end{align}
where 
\begin{align}\label{eq.ratio}
    \mathscr{C} := \frac{1}{2} + \frac{1}{2}\,
    \frac{(M + 1)^4\!\bigl(\sqrt{2M^2 - 2M + 1} - M\bigr)^2}
         {(M - 1)^4\!\bigl(2 + \sqrt{2M^2 - 2M + 1} - M\bigr)^2}.
\end{align}
\end{theorem}
\begin{proof}
    See Appendix~\ref{sec.app_proof_3}.
\end{proof}

\begin{remark}\label{rem:lowerbound}
Theorem~\ref{theo:lowerbound} focuses on the regime where $\tfrac{\sigma}{R}$ is sufficiently small, which corresponds to the most practical stochastic optimization settings. In this regime, the theorem guarantees that for any adequately small $\frac{\sigma}{R}$, it is always possible to select a parameter $\bar v \in \left[0, \tfrac{(M-1)n}{(n-1)M}\right]$ such that the proposed stepsize schedule outperforms the constant stepsize $\tfrac{2}{M+1}$ by a constant factor $\mathscr{C} < 1$. The constant $\mathscr{C}$ measures this improvement in terms of the optimal objective value associated with~\eqref{pep:one-step}.  
As shown in Figure~\ref{fig:c-vs-M}, $\mathscr{C}$ lies approximately in the interval $(0.92,1)$. Meanwhile, Figure~\ref{fig:upper-vs-v} plots $\sqrt{\mathscr{U}(M,n)}$ over  $(M,n)\in(1,20]\times ([100]\backslash\{1\})$ and demonstrates that the admissible region $\tfrac{\sigma}{R} < \sqrt{\mathscr{U}(M,n)}$ remains fairly large across a wide range of problem parameters. It is worth mentioning that when $M\approx 1$ (equivalently, $M\approx m$), the constant stepsize $\frac{2}{M+1}$ already produces an iterate that is nearly optimal within a single step along the negative gradient direction. In the stochastic setting of our interest, when $\frac{\sigma}{R}$ is small and $M\approx 1$, this one-step update remains nearly optimal, up to an additional $\sigma$-dependent error term in the distance to $x_*$. Consequently, when $\frac{\sigma}{R}$ is sufficiently small and $M\approx 1$, the constant stepsize policy becomes intrinsically difficult to outperform. This behavior is consistent with Figure~\ref{fig:c-vs-M}, where $\mathscr{C}\nearrow 1$ as $M\searrow 1$.

\begin{figure}[htbp]
    \centering
    \setlength{\abovecaptionskip}{2pt}
    \setlength{\belowcaptionskip}{-2pt}
    \begin{subfigure}[t]{0.4\textwidth}
        \centering
        \includegraphics[width=\textwidth]{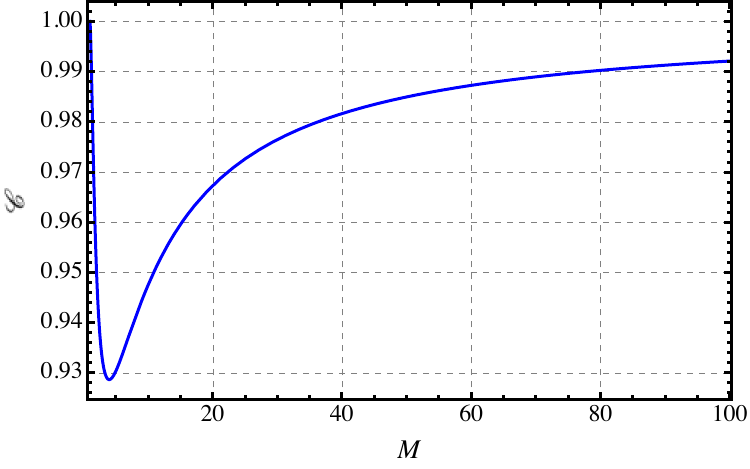}
        \caption{\small Dependence of $\mathscr{C}$ on $M$.}
        \label{fig:c-vs-M}
    \end{subfigure}
    \hfill
    \begin{subfigure}[t]{0.4\textwidth}
        \centering
        \includegraphics[width=\textwidth]{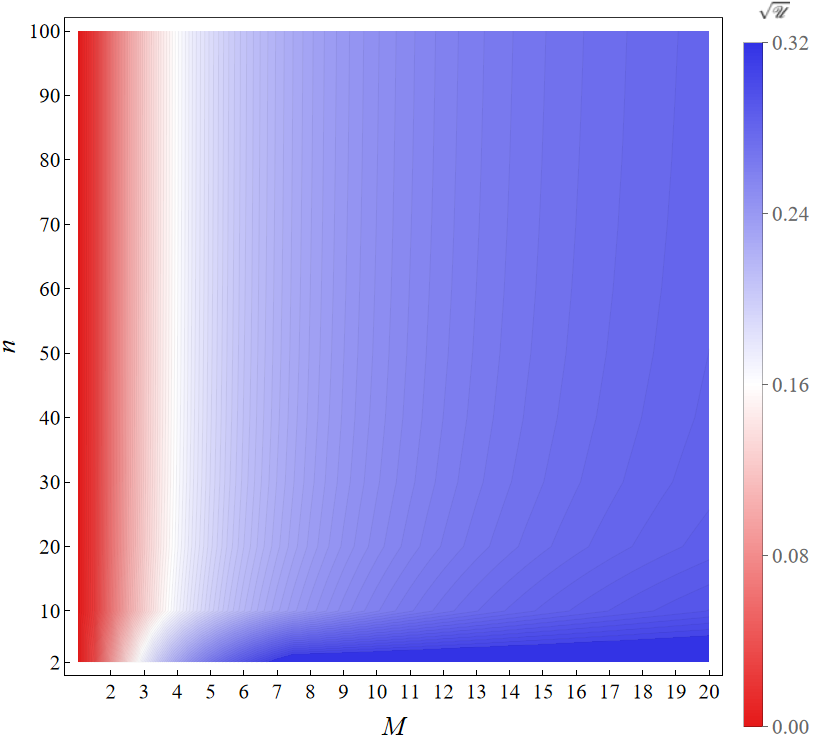}
        \caption{\small Dependence of $\sqrt{\mathscr{U}(M,n)}$ on $M$ and $n$.}
        \label{fig:upper-vs-v}
    \end{subfigure}

    \caption{\small 
    Illustration of the quantities $\mathscr{C}$ and $\sqrt{\mathscr{U}(M,n)}$ defined in Theorem~\ref{theo:lowerbound}.}
    \label{fig:auxiliary-small}
\end{figure}
\end{remark}

\section{Numerical Validation}\label{sec.numerical}
Under the assumption $m = 1$, this section presents numerical experiments that validate the theory and compare the proposed two-step stepsize schedule with the constant stepsize $\tfrac{2}{M+1}$. Although the direct evaluation of $h_{\mathrm{silver}}(v, R, \sigma)$ is difficult, as indicated by Theorem~\ref{theo:dual_feasibility} and Remark~\ref{rem:dual_feasibility}, the proposed schedule admits an explicit dual-feasible construction that provides a computable upper bound such as $h(v,\frac{\sigma}{R}) \geq \frac{h_{\mathrm{silver}}(v, R, \sigma)}{R^2}$ for any $v \in \left[0, \tfrac{(M-1)n}{(n-1)M}\right]$. For the constant-stepsize benchmark, the primal PEP formulation in~\eqref{pep:primal_prob} can be solved directly with the SDP solver \texttt{SDPT3}~\cite{tutuncu2003solving} to obtain the optimal objective value $h_{\mathrm{constant}}(R, \sigma)$, and we report its rescaled value $\tfrac{h_{\mathrm{constant}}(R, \sigma)}{R^2}$. By comparing $h(v, \tfrac{\sigma}{R})$ with $\tfrac{h_{\mathrm{constant}}(R, \sigma)}{R^2}$, the following interpretation holds:
\begin{itemize}
\item If $h(v,\frac{\sigma}{R}) < \frac{h_{\mathrm{constant}}(R, \sigma)}{R^2}$ for some $v \in \left[0, \tfrac{(M-1)n}{(n-1)M}\right]$, then~\eqref{eq.upperbound}--\eqref{eq.objfunc} imply that such $v$ also results in $h_{\mathrm{silver}}(v, R, \sigma) < h_{\mathrm{constant}}(R, \sigma)$. In this case, the proposed two-step schedule achieves better performance than the constant stepsize $\tfrac{2}{M+1}$.
\item If $h(v,\frac{\sigma}{R}) \geq \frac{h_{\mathrm{constant}}(R, \sigma)}{R^2}$ for all $v \in \left[0, \tfrac{(M-1)n}{(n-1)M}\right]$, then our proposed stepsize schedule may not be more competitive than the constant stepsize policy. As illustrated in Figure~\ref{fig.final}, this situation occurs only when $\tfrac{\sigma}{R}$ is not sufficiently small, which is consistent with the behavior described in Theorem~\ref{theo:lowerbound}.
\end{itemize}
Figure~\ref{fig.final} illustrates the comparison between the computable upper bound $h(v,\frac{\sigma}{R})$ and the rescaled constant-stepsize performance $\frac{h_{\mathrm{constant}}(R, \sigma)}{R^2}$, especially their changes depending on the relative noise level~$\frac{\sigma}{R}$, the smoothness parameter~$M$, and the number of stochastic supports~$n$.
According to Theorem~\ref{theo:lowerbound}, the proposed schedule outperforms the constant stepsize $\tfrac{2}{M+1}$ whenever $\tfrac{\sigma}{R} < \sqrt{\mathscr{U}(M, n)}$, and the value $\sqrt{\mathscr{U}(M, n)}$ depends explicitly on $M$ and $n$; see Figure~\ref{fig:upper-vs-v}. This prediction agrees with the numerical results shown in Figure~\ref{fig.final}. As shown by the first plot in the first row of Figure~\ref{fig.final}\textbf{(a)}, corresponding to $(M, n, \tfrac{\sigma}{R}) = (2, 2, 0.1)$, there is no value of $v\in\left[0, \tfrac{(M-1)n}{(n-1)M}\right]$ for which $h(v, \tfrac{\sigma}{R}) < \tfrac{h_{\mathrm{constant}}(R, \sigma)}{R^2}$. This observation matches the fact that $\tfrac{\sigma}{R} = 0.1$ is larger than $\sqrt{\mathscr{U}(2, 2)} \approx 0.0886$. Once $\tfrac{\sigma}{R}$ becomes smaller than $\sqrt{\mathscr{U}(2, 2)} \approx 0.0886$, as seen in the second through fourth plots in the first row of Figure~\ref{fig.final}\textbf{(a)}, there always exists a nonempty interval of $v$ for which $h\!\left(v, \tfrac{\sigma}{R}\right) < \tfrac{h_{\mathrm{constant}}(R, \sigma)}{R^2}$. This interval becomes larger as $M$ increases, which can be observed from a comparison across different rows of Figures~\ref{fig.final}\textbf{(a)}--\textbf{(b)}, considering different selections of $n\in\{2,100\}$. This trend is consistent with the monotonic growth of $\mathscr{U}(M, n)$ in~$M$ for general $n\in\mathbb{N}_{\geq 2}$, shown in Figure~\ref{fig:upper-vs-v}.

\begin{figure}[htbp]
    \centering
    \captionsetup[subfigure]{labelformat=empty}

    \begin{minipage}{\textwidth}
        \centering
        \begin{subfigure}[t]{0.23\textwidth}\centering
            \includegraphics[width=\textwidth]{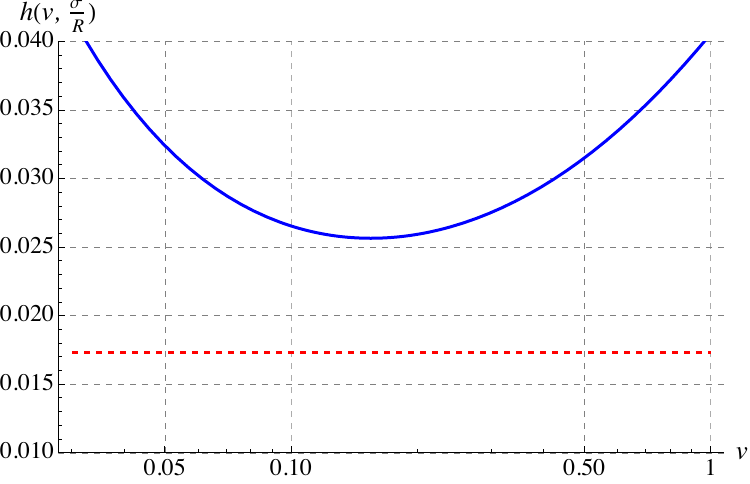}
        \end{subfigure}\hfill
        \begin{subfigure}[t]{0.23\textwidth}\centering
            \includegraphics[width=\textwidth]{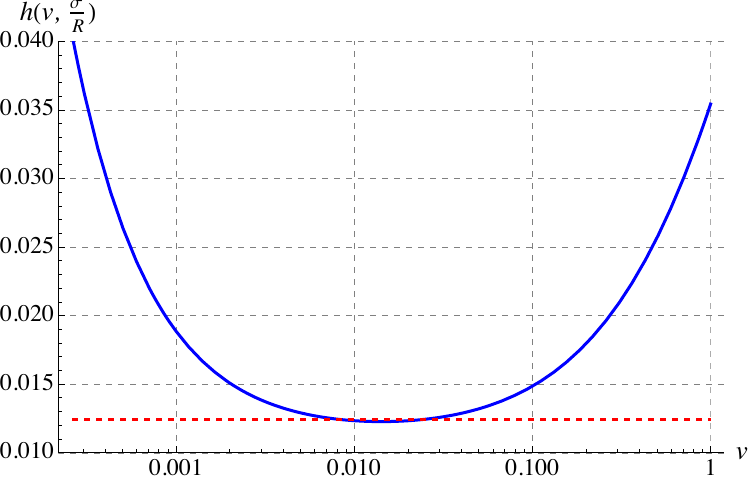}
        \end{subfigure}\hfill
        \begin{subfigure}[t]{0.23\textwidth}\centering
            \includegraphics[width=\textwidth]{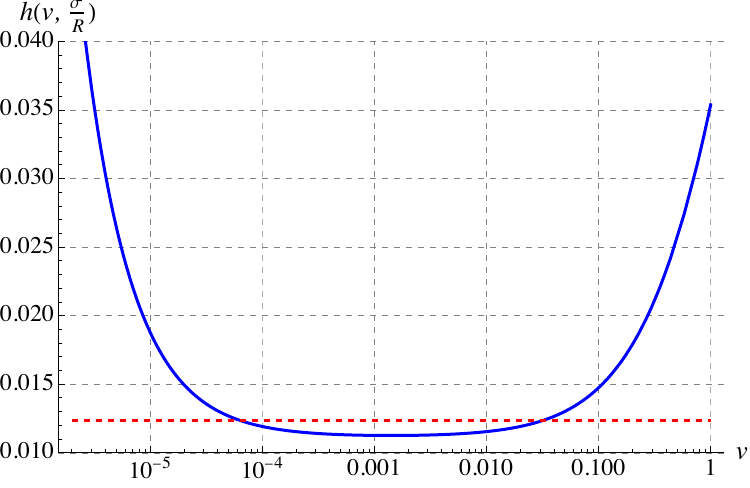}
        \end{subfigure}\hfill
        \begin{subfigure}[t]{0.23\textwidth}\centering
            \includegraphics[width=\textwidth]{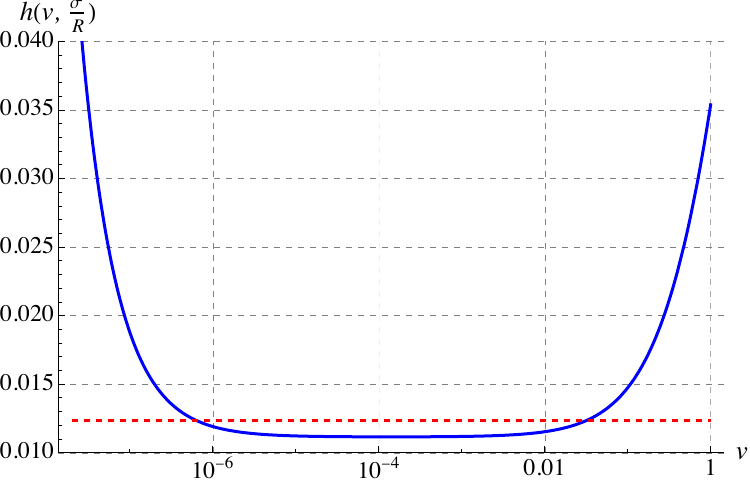}
        \end{subfigure}
    \end{minipage}

    \vspace{6pt}

    \begin{minipage}{\textwidth}
        \centering
        \begin{subfigure}[t]{0.23\textwidth}\centering
            \includegraphics[width=\textwidth]{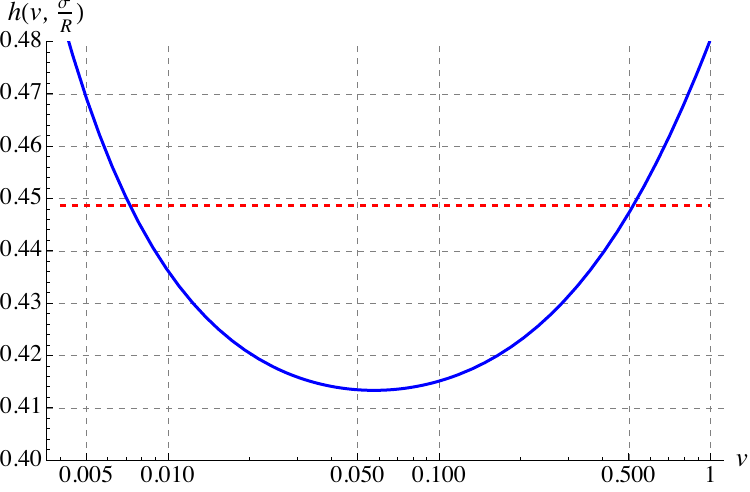}
        \end{subfigure}\hfill
        \begin{subfigure}[t]{0.23\textwidth}\centering
            \includegraphics[width=\textwidth]{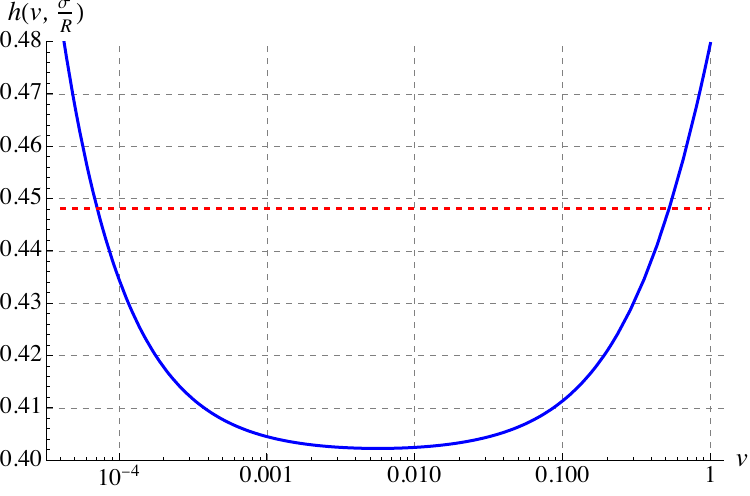}
        \end{subfigure}\hfill
        \begin{subfigure}[t]{0.23\textwidth}\centering
            \includegraphics[width=\textwidth]{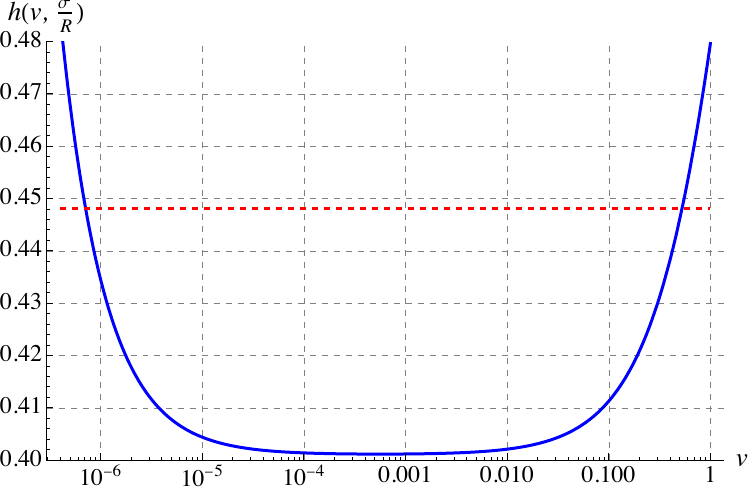}
        \end{subfigure}\hfill
        \begin{subfigure}[t]{0.23\textwidth}\centering
            \includegraphics[width=\textwidth]{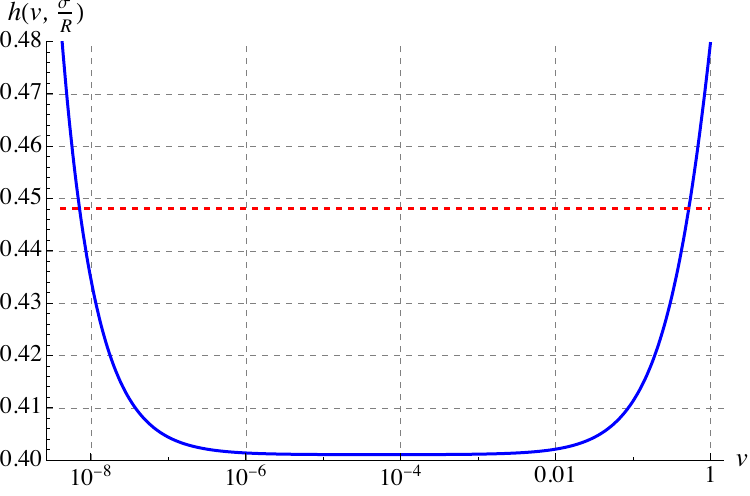}
        \end{subfigure}
    \end{minipage}

    \vspace{6pt}

    \begin{minipage}{\textwidth}
        \centering
        \begin{subfigure}[t]{0.23\textwidth}\centering
            \includegraphics[width=\textwidth]{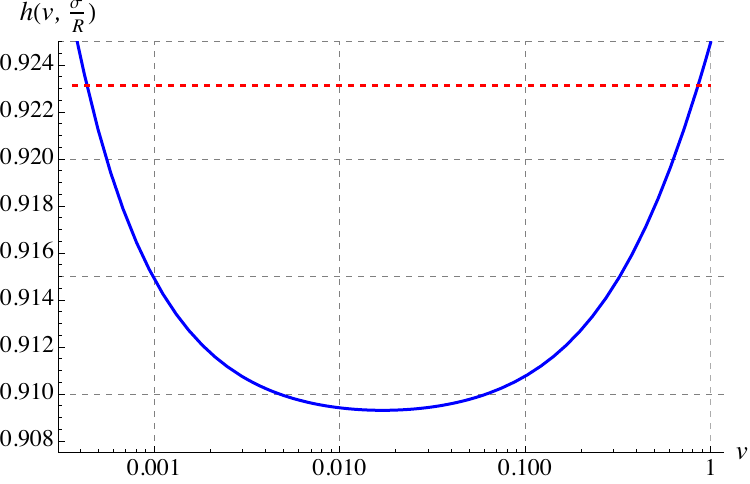}
        \end{subfigure}\hfill
        \begin{subfigure}[t]{0.23\textwidth}\centering
            \includegraphics[width=\textwidth]{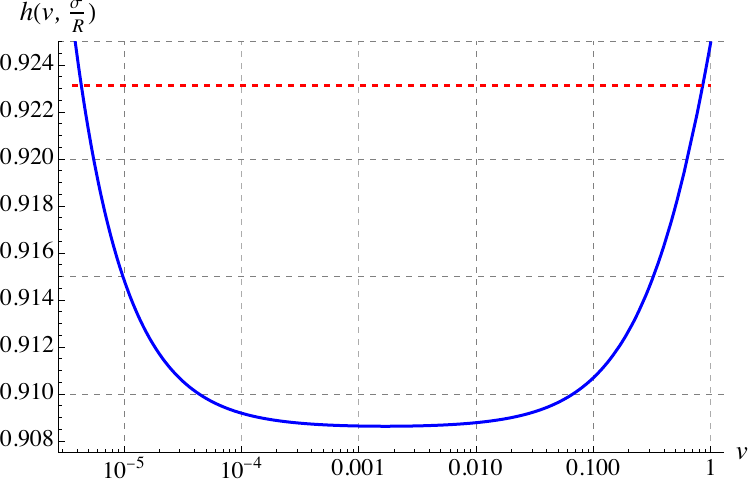}
        \end{subfigure}\hfill
        \begin{subfigure}[t]{0.23\textwidth}\centering
            \includegraphics[width=\textwidth]{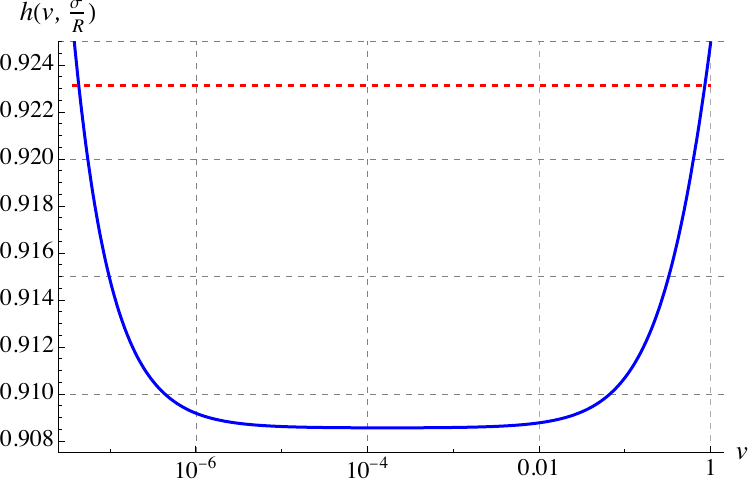}
        \end{subfigure}\hfill
        \begin{subfigure}[t]{0.23\textwidth}\centering
            \includegraphics[width=\textwidth]{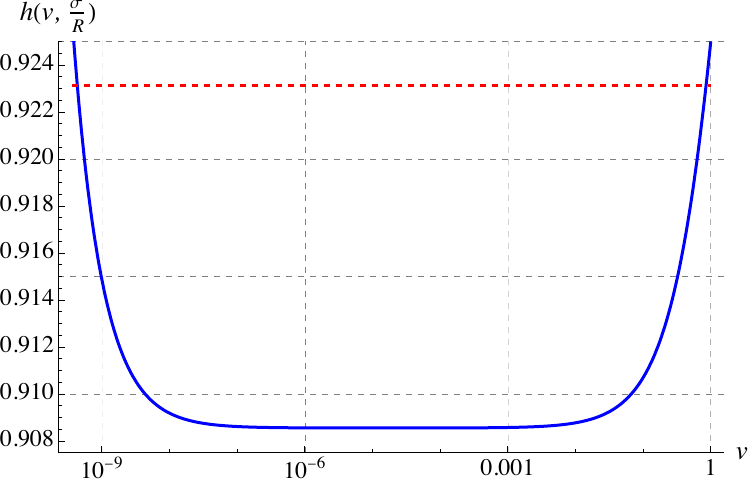}
        \end{subfigure}
    \end{minipage}

    \vspace{4pt}
    \centering
    \textbf{(a)}\; Fixing $n=2$ with different $M$ values; top row: $M=2$, middle row: $M=10$, bottom row: $M=100$.


    \vspace{10pt}

    \begin{minipage}{\textwidth}
        \centering
        \begin{subfigure}[t]{0.23\textwidth}\centering
            \includegraphics[width=\textwidth]{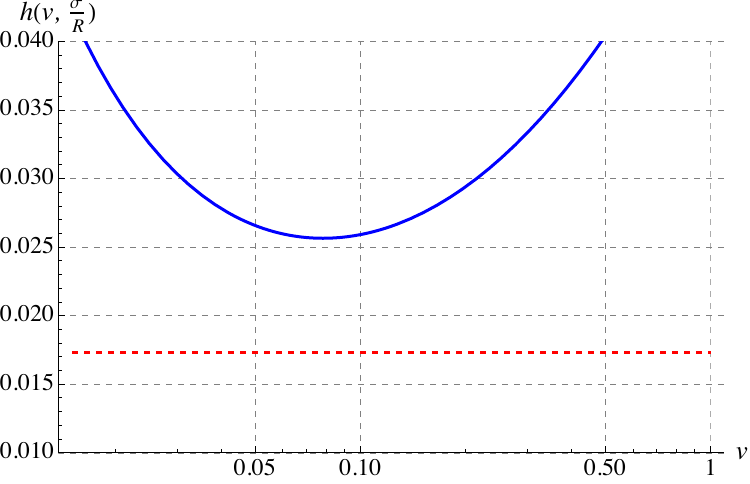}
        \end{subfigure}\hfill
        \begin{subfigure}[t]{0.23\textwidth}\centering
            \includegraphics[width=\textwidth]{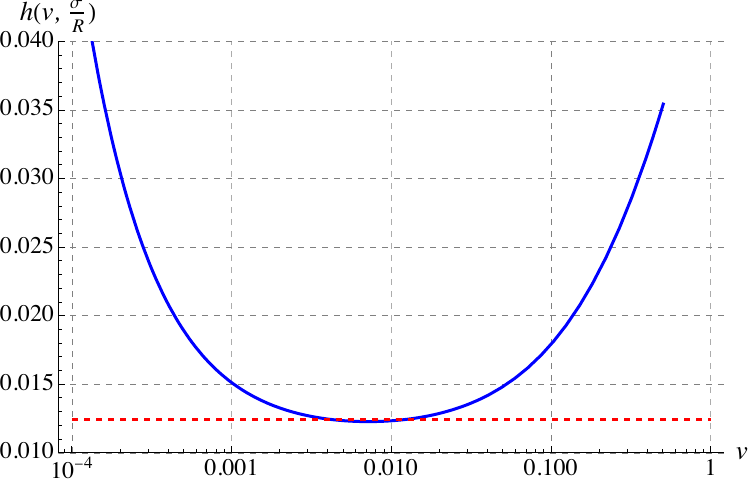}
        \end{subfigure}\hfill
        \begin{subfigure}[t]{0.23\textwidth}\centering
            \includegraphics[width=\textwidth]{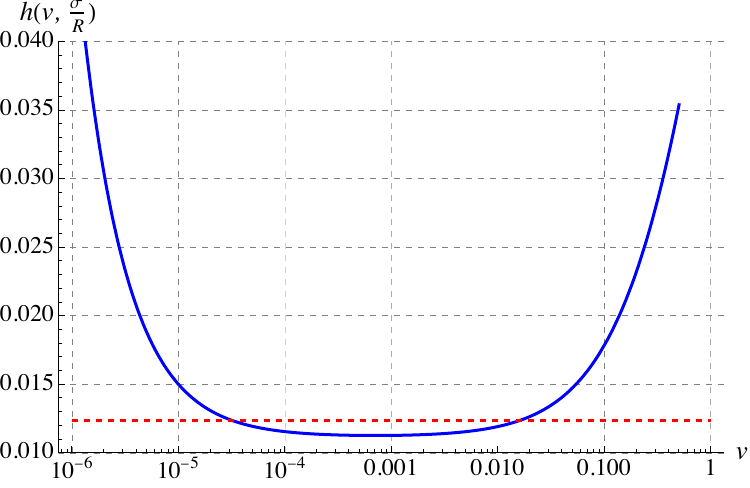}
        \end{subfigure}\hfill
        \begin{subfigure}[t]{0.23\textwidth}\centering
            \includegraphics[width=\textwidth]{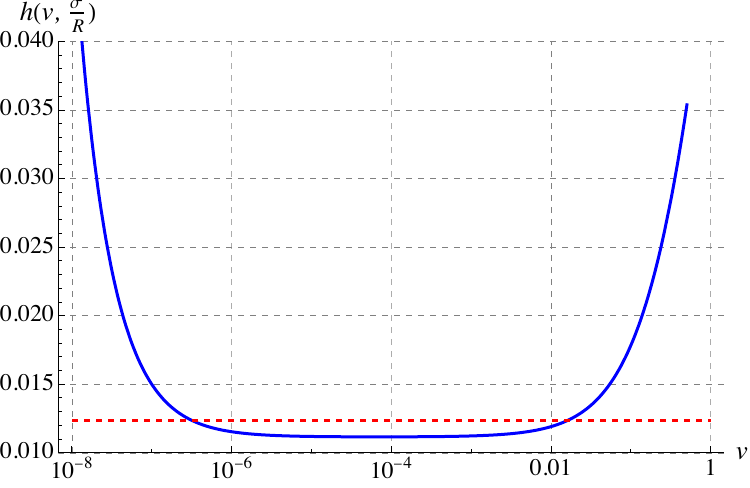}
        \end{subfigure}
    \end{minipage}

    \vspace{6pt}

    \begin{minipage}{\textwidth}
        \centering
        \begin{subfigure}[t]{0.23\textwidth}\centering
            \includegraphics[width=\textwidth]{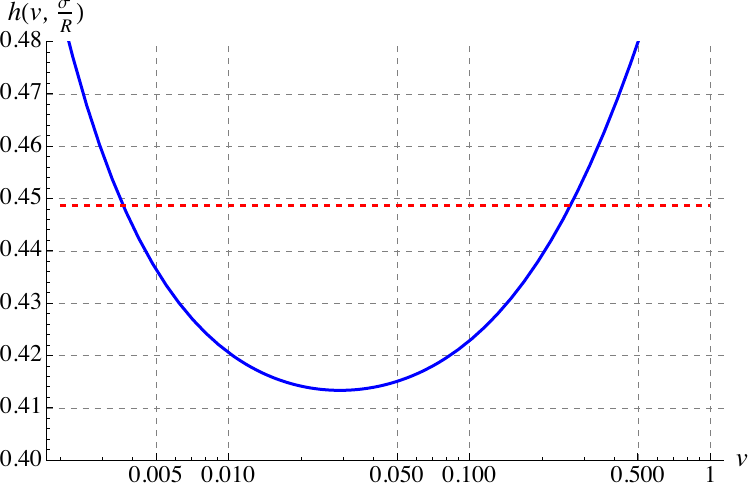}
        \end{subfigure}\hfill
        \begin{subfigure}[t]{0.23\textwidth}\centering
            \includegraphics[width=\textwidth]{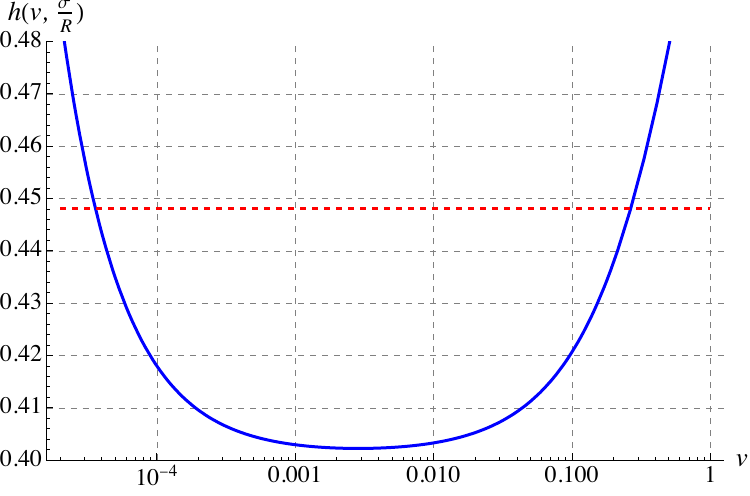}
        \end{subfigure}\hfill
        \begin{subfigure}[t]{0.23\textwidth}\centering
            \includegraphics[width=\textwidth]{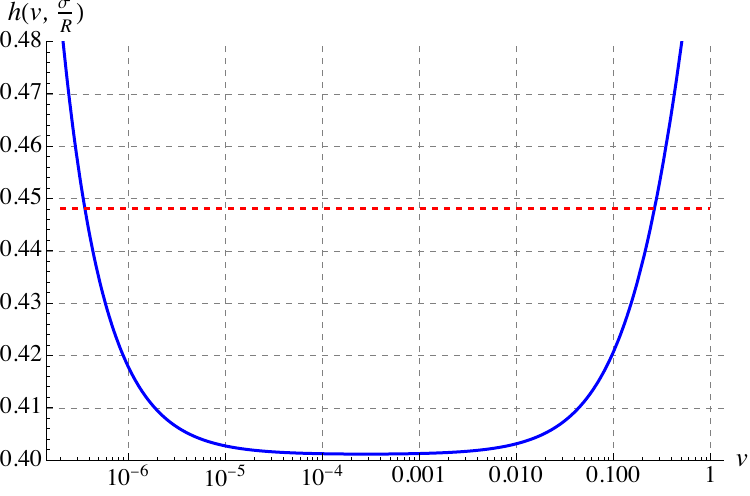}
        \end{subfigure}\hfill
        \begin{subfigure}[t]{0.23\textwidth}\centering
            \includegraphics[width=\textwidth]{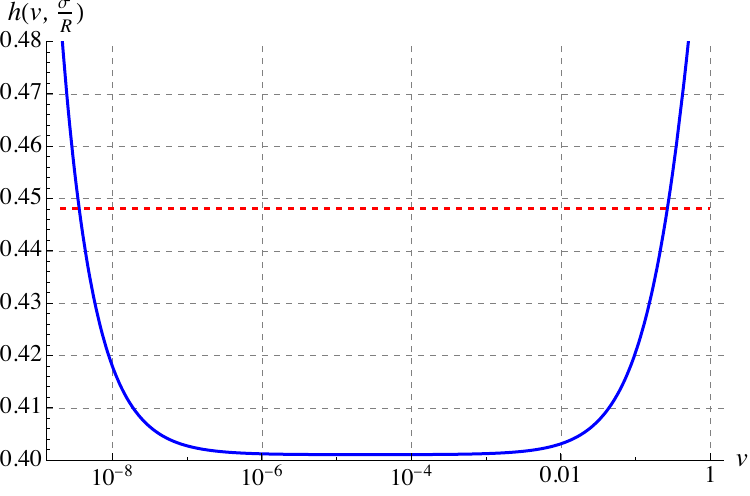}
        \end{subfigure}
    \end{minipage}

    \vspace{6pt}

    \begin{minipage}{\textwidth}
        \centering
        \begin{subfigure}[t]{0.23\textwidth}\centering
            \includegraphics[width=\textwidth]{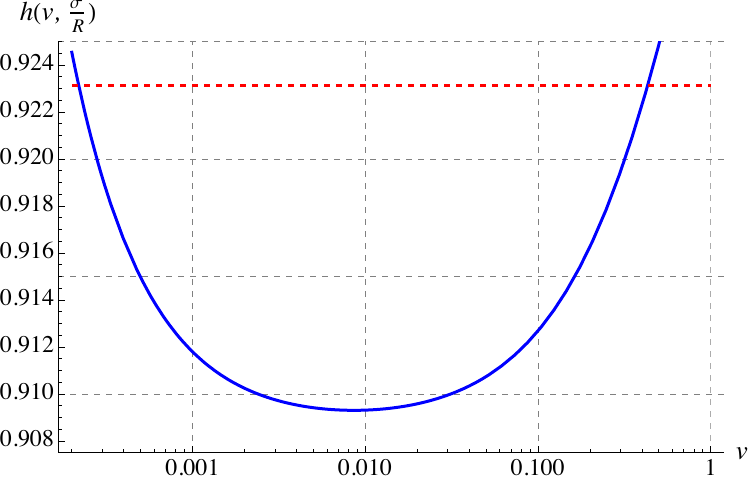}
        \end{subfigure}\hfill
        \begin{subfigure}[t]{0.23\textwidth}\centering
            \includegraphics[width=\textwidth]{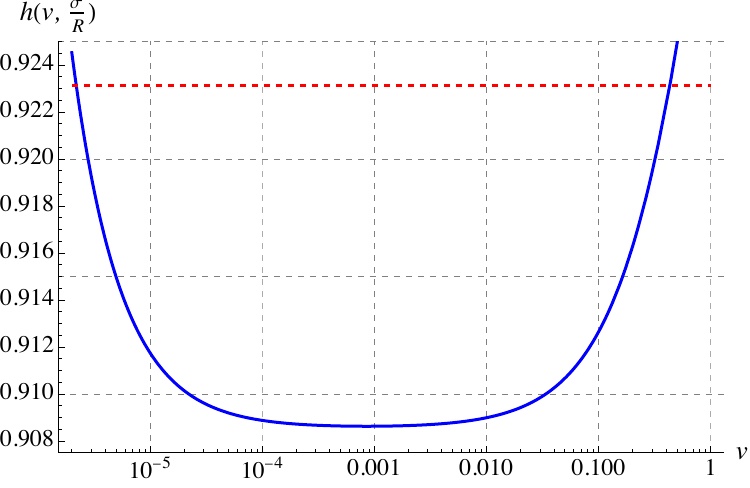}
        \end{subfigure}\hfill
        \begin{subfigure}[t]{0.23\textwidth}\centering
            \includegraphics[width=\textwidth]{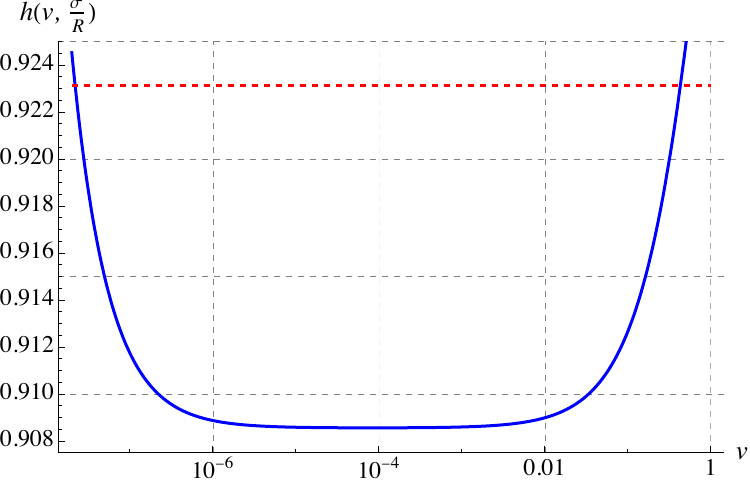}
        \end{subfigure}\hfill
        \begin{subfigure}[t]{0.23\textwidth}\centering
            \includegraphics[width=\textwidth]{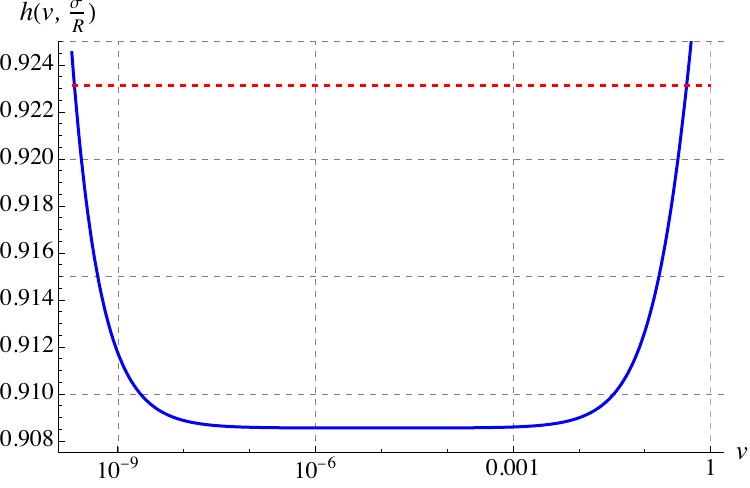}
        \end{subfigure}
    \end{minipage}

    \vspace{4pt}
    \centering
    \textbf{(b)}\;Fixing $n=100$ with different $M$ values; top row: $M=2$, middle row: $M=10$, bottom row: $M=100$.

    \caption{
        Comparison of $h(v, \frac{\sigma}{R})$ and $\frac{h_{\mathrm{constant}}(R, \sigma)}{R^2}$. 
        Each column corresponds to different relative noise levels such as
        $\frac{\sigma}{R} = 0.1$, $0.01$, $0.001$, and $0.0001$ from left to right. In each plot, the blue curve denotes $h(v, \frac{\sigma}{R})$ as a function of $v\in\left[0, \tfrac{(M-1)n}{(n-1)M}\right]$, while the red dashed lines represents $\frac{h_{\mathrm{constant}}(R, \sigma)}{R^2}$, the rescaled optimal objective value of PEP~\eqref{pep:primal_prob} 
        under the constant stepsize $\tfrac{2}{M+1}$.
    }\label{fig.final}
\end{figure}

\section{Conclusion}\label{sec.conclusion}
In this work, we explore a two-step stepsize schedule for stochastic gradient methods within a tractable stochastic PEP framework. The framework extends deterministic performance estimation to a stochastic setting that considers all stochastic gradient estimates are unbiased estimators with bounded variance and generated from a finite number of supports. A feasible solution to the associated dual problem is constructed, yielding our stepsize rule that generalizes the classical two-step silver stepsize policy in~\cite{altschuler2025acceleration} to the stochastic setting. The proposed schedule has been proven to achieve strictly better performance than the constant stepsize $\frac{2}{M+m}$ whenever the initial optimality gap dominates the noise in stochastic gradient estimates, which is a natural condition that commonly holds in practical stochastic optimization. As discussed in Remark~\ref{rem:dual_feasibility}, when the proposed schedule is applied and the resulting iterate moves closer to the optimum in expectation, applying the two-step scheme again at the new starting point suggests the use of a larger value of $v$ for improved convergence performance in terms of $h\left(v, \tfrac{\sigma}{R}\right)$ in~\eqref{eq.objfunc}. This increase in $v$ leads to smaller stepsizes, as illustrated in Figure~\ref{fig:step-vs-v}. This observation matches our common knowledge that SG with diminishing stepsizes may achieve exact convergence in expectation. Meanwhile, how to extend our current two-step stepsize schedule to a multi-step scheme remains as an open question.

\bibliographystyle{plain}
\bibliography{reference}

\newpage

\section{Appendix}
\subsection{Proof of Theorem~\ref{lem:root-interval}}\label{sec.app_proof_1}

\begin{proof}[Proof of Theorem~\ref{lem:root-interval}]
We start with showing the uniqueness of $\alpha^*$ and then prove the uniqueness of $\beta^*$. \\
\textbf{The uniqueness of $\alpha^*$.} With predetermined parameters $(M,n,v)$ described in the statement, we can rewrite~\eqref{eq.stepsize.alpha} as $h(\alpha) =0$, where for any $\alpha\in\R$,
\begin{equation}\label{eq:h_func}
h(\alpha) := (M - 1)Mn\alpha^3 - (M + 1)(M - 2)n\alpha^2 - (4n + (M-1)(n-1)v)\alpha + 2n. 
\end{equation}
Consider the case $n=1$. $h(\alpha)$ is a cubic function in $\alpha$, where $h(\alpha) = 0$ (or equivalently, \eqref{eq.stepsize.alpha}) has closed-form solutions. By solving $h(\alpha) = 0$, we find out that within the interval $\left[\frac{1}{M},1\right)$, there is a unique
\begin{equation}\label{eq:silver.alpha}
    \alpha^* = \frac{2}{\,1 + \sqrt{\,1 - 2M + 2M^2\,}}
\end{equation}
satisfying $h(\alpha^*) = 0$. Therefore, the uniqueness of $\alpha^*\in\left[\frac{1}{M},1\right)$ holds when $n=1$.

On the other hand, let's consider the case when $n\in\N_{>1}$. In this scenario, it directly follows~\eqref{eq:h_func} that
\begin{equation}\label{eq:h_function_val}
h\left(1\right)
    = -(M - 1)(n-1)v \quad \text{and} \quad
h\left(\tfrac{1}{M}\right)
    = \frac{(M - 1)\left(M v + n (M - 1 - M v)\right)}{M^2}.
\end{equation}
Note that for any $(M,n,v)\in\R_{>1}\times\N_{>1}\times\left[0, \tfrac{(M-1)n}{(n-1)M}\right]$, we have $h(1) = -(M - 1)(n-1)v \leq 0$. Next, we are going to show that $h\left(\tfrac{1}{M}\right) \geq 0$. To this end, let's define
\begin{align}\label{eq:phi_def}
\phi(v) := M v + n (M - 1 - M v) = v(M - Mn) + n(M - 1),
\end{align}
which is an affine function of $v$ and monotonically decreases as $v$ increasing because of $M - Mn < 0$ for $(M,n)\in\R_{>1}\times\N_{>1}$. Therefore, for any $v\in\left[0, \tfrac{(M-1)n}{(n-1)M}\right]$, we have from~\eqref{eq:phi_def} that
\begin{align*}
\phi(v) \geq \phi\left(\tfrac{(M-1)n}{(n-1)M}\right) = 0,
\end{align*}
which further implies from~\eqref{eq:h_function_val} that
\begin{equation}\label{eq:h_endpoint_val}
h\left(\tfrac{1}{M}\right) = \frac{(M-1)\phi(v)}{M^2} \geq 0.
\end{equation}
Since function $h:\R\to\R$ is continuous with $h\left(\tfrac{1}{M}\right) \geq 0$ and $h(1) \leq 0$, the intermediate value theorem guarantees the existence of some $\alpha$ within $\left[\tfrac{1}{M},1\right]$ satisfying $h(\alpha) = 0$ (or equivalently,~\eqref{eq.stepsize.alpha}).

Then we prove the uniqueness of $\alpha^*\in\left[\tfrac{1}{M},1\right)$ satisfying $h(\alpha^*) = 0$. Based on~\eqref{eq:h_func}, let's consider the derivative function $h':\R\to\R$, where for any $\alpha\in\R$,
\begin{equation}\label{eq:h_derivative}
h'(\alpha)
= 3(M-1)Mn\alpha^{2}
- 2(M+1)(M-2)n\alpha
- 4n - (M-1)(n-1)v .
\end{equation}
Following~\eqref{eq:h_derivative}, we know
\begin{equation}\label{eq:h_derivative_sign}
h'(1)=(M-1)\bigl(Mn-(n-1)v\bigr)>0
\quad\text{and}\quad
h'\left(\tfrac{1}{M}\right)
=-\,\frac{(M-1)\bigl(n + 2Mn + M(n-1)v\bigr)}{M} < 0,
\end{equation}
where the inequalities are from parameter selections $(M,n,v)\in\R_{>1}\times\N_{>1}\times\left[0, \tfrac{(M-1)n}{(n-1)M}\right]$. Because $3(M-1)Mn > 0$, $h'(\alpha)$ in~\eqref{eq:h_derivative} is a strongly convex quadratic function. Moreover, it follows~\eqref{eq:h_derivative_sign} that there exists
$r\in\left(\frac{1}{M},1\right)$ such that
\begin{equation}\label{eq:h_monotonicity}
h'(r) = 0, \quad
h'(x) < 0 \ \ \forall x\in\left[\tfrac{1}{M},r\right),
\quad\text{and}\quad
h'(x) > 0 \ \ \forall x\in(r,1].
\end{equation}
Hence, function $h:\R\to\R$ is strictly monotonically decreasing on $\left[\frac{1}{M},r\right]$ and strictly monotonically increasing on $[r,1]$. Depending on whether $v=0$ or not, there are two cases as follows:\\
\textbf{Case (i).} When $v\neq 0$, we have $h(1)<0$ from~\eqref{eq:h_function_val}. Together with~\eqref{eq:h_endpoint_val} and the monotonicity revealed in~\eqref{eq:h_monotonicity}, we have that there exists a unique $\alpha^*\in\left[\frac{1}{M},1\right)$ satisfying $h(\alpha^*) = 0$.\\
\textbf{Case (ii).} When $v=0$, \eqref{eq:h_function_val} implies $h\left(\frac{1}{M}\right) > 0$ and $h(1)=0$. In addition, by the help of~\eqref{eq:h_monotonicity}, in this case, there also exists a unique $\alpha^*\in\left[\frac{1}{M},1\right)$ satisfying $h(\alpha^*) = 0$.

Combining both cases discussed above, we complete the proof that demonstrates the uniqueness of $\alpha^*$.

Based on the uniqueness of $\alpha^*\in\left[\frac{1}{M},1\right)$ such that $\alpha=\alpha^*$ performs as a solution of~\eqref{eq.stepsize.alpha}, we are going to show the uniqueness of $\beta^*\in\left(\frac{\alpha^*}{M},\frac{M+1}{2M}\right)$ with $(\alpha^*,\beta^*)$ satisfying~\eqref{eq.stepsize.beta}. \\
\textbf{The uniqueness of $\beta^*$.} After fixing $\alpha = \alpha^*$, we denote~\eqref{eq.stepsize.beta} in the format of $q(\beta) = 0$, where $\beta$ is the only changing variable and function $q:\R\to\R$ also relies on parameters $(M,n,v,\alpha^*)$, i.e.,
\begin{align}
q(\beta) = \ &M n \left((-1 + M)^2 n + (1 + M)^2 (-1 + n) v \right) 
\left( (-1+M)\alpha^* + (1+M) \right)\beta^2 \label{eq.q_func} \\
&+ 2M \Bigl[
       2(-1+M)Mn(-n + (n-1)v) (\alpha^*)^2 \nonumber{} \\ 
&\quad\quad\quad\  
    + (1+M)(-n + (n-1)v)\,((3 - M) n + (-1 + M) (-1 + n) v) \alpha^* \nonumber{} \\
&\quad\quad\quad\  
    + n\left( n - (M - 4) M n - (1 + M) (3 + M) (-1 + n) v \right) 
   \Bigr]\beta \nonumber{} \\
&+ \Bigl[
     -2 M (-1 + M^2) n (-n + (n - 1) v)(\alpha^*)^2 \nonumber{} \\ 
&\quad\ \ 
   + \left( 
        2 (2 + M + 2 M^2 - M^3) n^2 + (-3 + M) (1 + M)^2 (-1 + n) n v + (1 - 
    M) (1 + M)^2 (-1 + n)^2 v^2  
     \right)\alpha^* \nonumber{} \\
&\quad\ \ 
   + 2(1+M)^2 n(-n + (n - 1)v)
   \Bigr]. \nonumber{}
\end{align}
When $n=1$, after substituting the value of $\alpha^*$ (see~\eqref{eq:silver.alpha}) into $q(\beta) = 0$ (or equivalently,~\eqref{eq.stepsize.beta}),  
the resulting quadratic equation of $\beta$ only has one solution within the interval $\left(\tfrac{\alpha^*}{M}, \tfrac{M + 1}{2M}\right)$, which is
\begin{equation}\label{eq.silver.beta}
    \beta^* = \frac{2}{\,1 + 2M - \sqrt{\,2M^2 - 2M + 1\,}}.
\end{equation}
Therefore, the uniqueness of $\beta^*\in\left(\tfrac{\alpha^*}{M}, \tfrac{M + 1}{2M}\right)$ holds when $n=1$.

On the other hand, let's consider the case of $n\in\N_{>1}$. For any given parameters $(M,n,v)\in\R_{>1}\times\N_{>1}\times\left[0, \tfrac{(M-1)n}{(n-1)M}\right]$ and $\alpha^*\in\left[\frac{1}{M},1\right)$ solved from~\eqref{eq.stepsize.alpha}, let's evaluate the value of $q(\beta)$ at $\beta = \frac{\alpha^*}{M}$ and $\beta = \frac{M+1}{2M}$, respectively.
First, it directly follows the definition of function $q:\R\to\R$ that
\begin{equation}\label{eq.q_func_endpoint}
q\left(\tfrac{M+1}{2M}\right) = \frac{(M-1)^2 n}{4M} \Bigl((M - 1)^2 n + (n - 1)(1 + M)^2 v \Bigr)\bigl(1+M + (M-1)\alpha^*\bigr) > 0,
\end{equation}
where the inequality is from the parameter setting of $(M,n,v,\alpha^*)$. Meanwhile, by the help of cubic equation~\eqref{eq.stepsize.alpha}, we can formulate $q\left(\tfrac{\alpha^*}{M}\right)$ as a quadratic form of $\alpha^*$, i.e.,
\begin{equation}\label{eq.psi_func}
q\left(\tfrac{\alpha^*}{M}\right) = \frac{\psi(\alpha^*)}{M^2} \quad \text{and} \quad \psi(\alpha^*) = P_2(v)\,(\alpha^*)^2 + P_1(v)\,\alpha^* + P_0(v),
\end{equation}
where
\begin{equation}\label{eq.P_funcs}
\begin{aligned}
    &P_2(v) = 2 M^2 (-1 + M^2) (-1 + n)^2v^2 -2 (1 + 2 M - 3 M^3 - M^4 + M^5) (-1 + n) n v \\
    &\quad\quad\quad\quad+ 2 (-1 + M)^2 (-1 + 2 M^2 + M^3) n^2, \\
    &P_1(v) = -(1 + M + 2 M^2 - 6 M^3 + M^4 + M^5) (-1 + n)^2 v^2 \\
    &\quad\quad\quad\quad + (3 + 11 M + 12 M^2 - 16 M^3 - 3 M^4 + M^5) (-1 + n) n v \\
    &\quad\quad\quad\quad + 2 (2 - 4 M - 3 M^2 + 5 M^3 + M^4 - M^5) n^2,\\
    \text{and} \quad &P_0(v) = 2 (-1 - 2 M - 4 M^2 + 2 M^3 + M^4) (-1 + n) nv -2 (1 - 2 M - 2 M^2 + 2 M^3 + M^4) n^2.
\end{aligned}
\end{equation}
Next, we are going to show that $q\left(\frac{\alpha^*}{M}\right) < 0$. From~\eqref{eq.psi_func}, we know $q\left(\frac{\alpha^*}{M}\right)$ shares the same sign as $\psi(\alpha^*)$. To analyze the sign of $\psi(\alpha^*)$ (see~\eqref{eq.psi_func}), we first focus on its leading coefficient $P_2(v)$.  
By~\eqref{eq.P_funcs}, we obtain
\begin{equation}\label{eq.P2_vals}
\begin{aligned}
P_2(0) &= 2(-1+M)^2(-1+2M^2+M^3)n^2 > 0, \\
\text{and}\quad
P_2\!\left(\tfrac{(M - 1)n}{(n - 1)M}\right) 
&= \frac{2(-1+M)^2 (1 + M + 3M^2 + 3M^3) n^2}{M} > 0,
\end{aligned}
\end{equation}
where the inequalities follows $(M,n)\in\R_{>1}\times\N_{>1}$. Thus, $P_2(v)$ is strictly positive at $v=0$ and $v=\tfrac{(M-1)n}{(n-1)M}$.  

Moreover, we notice from~\eqref{eq.P_funcs} that $P_2(v)$ is a strongly convex quadratic function of $v$. Therefore, to conclude $P_2(v) > 0$ for all $v\in\left[0, \tfrac{(M-1)n}{(n-1)M}\right]$, we only need to show that $P_2(v_{\min}) > 0$ when $v_{\min}\in\left[0, \tfrac{(M-1)n}{(n-1)M}\right]$, where $v_{\min}$ is the unique minimizer of the strongly convex quadratic function $P_2:\R\to\R$. 

By~\eqref{eq.P_funcs}, we directly have
\begin{equation}\label{eq.vmin}
v_{\min} = \frac{2(M^5 - M^4 - 3 M^3 + 2 M + 1) (n-1) n}{4M^2(M^2-1)(n-1)^2} = \frac{(M^3-M^2-2M-1)n}{2M^2(n-1)},
\end{equation}
which implies that $v_{\min} \in \left[0, \tfrac{(M-1)n}{(n-1)M}\right]$ only when the following two conditions hold simultaneously:
\begin{align}\label{L3.1.ieq.1}
M^3-M^2-2M-1 \;\geq\; 0 \quad &\text{(guaranteeing $v_{\min}\geq 0$)}, \\
\label{L3.1.ieq.2}
\text{and}\quad \quad -1 - 3M^2 + M^3 \;\leq\; 0 \quad &\text{(guaranteeing $v_{\min} \leq \tfrac{(M - 1)n}{(n - 1)M}$)}.
\end{align}

That is, \eqref{L3.1.ieq.1} specifies a range of $M$ for which $v_{\min} \geq 0$, and \eqref{L3.1.ieq.2} specifies a range of $M$ for which $v_{\min} \leq \tfrac{(M - 1)n}{(n - 1)M}$.  
Therefore, when $M$ lies in the overlap of these two ranges, $v_{\min} \in \left[0, \tfrac{(M-1)n}{(n-1)M}\right]$. Note that~\eqref{L3.1.ieq.1} and~\eqref{L3.1.ieq.2} both describe cubic functions of $M$. By applying the cubic root formula~\cite{MR945393}, we know that $v_{\min} \in \left[0, \tfrac{(M-1)n}{(n-1)M}\right]$ when
\begin{align*}
M \in \left[\tfrac{1}{3}
+ \left(\tfrac{47}{54} + \tfrac{\sqrt{93}}{18}\right)^{\!1/3}
+ \left(\tfrac{47}{54} - \tfrac{\sqrt{93}}{18}\right)^{\!1/3}, 1
+ \left(\tfrac{3+\sqrt{5}}{2}\right)^{\!1/3}
+ \left(\tfrac{3-\sqrt{5}}{2}\right)^{\!1/3}\right] \subset [2,4].
\end{align*}
In this situation, we further have from~\eqref{eq.P_funcs} and~\eqref{eq.vmin} that
\begin{align}\label{L3.1.ieq.3}
P_2(v_{\min}) 
= \frac{(M-1)(M+1)(-1 - 4 M - 2 M^2 - 10 M^3 + 3 M^4 + 6 M^5 - M^6)n^2}{2M^2}.
\end{align}
Notice that for any $M\in[2,4]$, 
\begin{align*}
    &-1 - 4 M - 2 M^2 - 10 M^3 + 3 M^4 + 6 M^5 - M^6 \\
    =& -M^5(M-4) + (M-2)(2M^4+7M^3+4M^2+6M+8)+15 > 0,
\end{align*}
which, together with~\eqref{L3.1.ieq.3}, further implies that $P_2(v_{\min}) > 0$ whenever $v_{\min} \in \left[0, \tfrac{(M-1)n}{(n-1)M}\right]$. Thus, $P_2(v) > 0$ for all $v \in \left[0, \tfrac{(M-1)n}{(n-1)M}\right]$.

Next, we focus on $P_0(v)$. It's known from~\eqref{eq.P_funcs} that $P_0(v)$ is an affine function of $v$. Meanwhile,
\begin{align*}
P_0(0) &= -2(1-2M-2M^2+2M^3+M^4)n^2 = -2n^2(M-1)(M^3 + 3M^2 + M - 1) < 0, \\
\text{and} \quad P_0\!\left(\tfrac{(M - 1)n}{(n - 1)M}\right) &= -\tfrac{2(-1 - 4M^2 + 4M^3 + M^4)n^2}{M} = -\frac{2n^2(M-1)}{M}\cdot(M^3 + 5M^2 + M + 1) < 0,
\end{align*}
where the inequalities are from $M\in\R_{>1}$. Thus, $P_0(v) < 0$ for all $v \in \left[0, \tfrac{(M-1)n}{(n-1)M}\right]$.

Since $P_2(v) > 0$ for all $v \in \left[0, \tfrac{(M-1)n}{(n-1)M}\right]$, from~\eqref{eq.psi_func}, we know the function $\psi:\R\to\R$ is strongly convex. 
Moreover, \eqref{eq.psi_func} also implies $\psi(0) = P_0(v) < 0$ for any $v \in \left[0, \tfrac{(M-1)n}{(n-1)M}\right]$. Next, we are going to show that $\psi(1) \leq 0$ as well. By~\eqref{eq.psi_func} and~\eqref{eq.P_funcs}, a direct calculation leads to, for any $v\in\left[0, \tfrac{(M-1)n}{(n-1)M}\right]$,
\begin{align*}
\psi(1)&= P_2(v) + P_1(v) + P_0(v) \\
&= (-1+M)(-1+n)\,v\left( n(-M^4 - 6M^2 - 2M + 1) + (n-1)v(-M^4 + 6M^2 + 2M + 1)\right) \\
&= (-1+M)(-1+n)\,v\left( (-M^4+1)(n + (n-1)v) - (6M^2 + 2M)(n-(n-1)v)\right) \leq 0,
\end{align*}
where the inequality comes from $(M,n,v)\in\R_{>1}\times\N_{>1}\times\left[0, \tfrac{(M-1)n}{(n-1)M}\right]$. Combining the fact of $\psi(0) < 0$, $\psi(1) \leq 0$, and $\psi:\R\to\R$ being strongly convex, we conclude $\psi(\alpha^*) < 0$ whenever $\alpha^*\in\left[\frac{1}{M},1\right)$. Moreover, by~\eqref{eq.psi_func}, it directly holds that $q\left(\frac{\alpha^*}{M}\right) = \frac{\psi(\alpha^*)}{M^2} < 0$. Since $q:\R\to\R$ is a quadratic function (from~\eqref{eq.stepsize.beta}) satisfying
\[
q\left(\tfrac{\alpha^*}{M}\right) < 0 \quad \text{and} \quad q\left(\tfrac{M+1}{2M}\right) > 0,
\]
where the second inequality is from~\eqref{eq.q_func_endpoint}, the intermediate value theorem ensures the existence of a unique
$\beta^* \in \left(\tfrac{\alpha^*}{M}, \tfrac{M+1}{2M}\right)$ satisfying $q(\beta^*) = 0$ (or equivalently, \eqref{eq.stepsize.beta}).

Combining both cases ($n=1$ and $n\in\N_{>1}$) discussed above, we complete the proof of the uniqueness of $\beta^*$.

Therefore, we conclude the statement.
\end{proof}

\subsection{Proof of Theorem~\ref{theo:dual_feasibility}}\label{sec.app_proof_2}

To start with, we introduce the next two lemmas that would be useful in the proof of Theorem~\ref{theo:dual_feasibility}.

\begin{lemma}\label{lem:lambda10}
For any parameters $(M,n,v)\in\R_{>1}\times\N\times\left[0, \tfrac{(M-1)n}{(n-1)M}\right]$, with the stepsize schedule $(\alpha^*,\beta^*)$ described by the statement of Theorem~\ref{lem:root-interval}, we have $r(\beta^*) \leq 0$, where $r:\R\to\R$ is a quadratic function defined as
\begin{equation}\label{eq.r_func}
\begin{aligned}
r(\beta) = \ &(2 M n \alpha^* + M v \alpha^* - (n-1) v \alpha^* - M n v \alpha^*) + 2 M (n + n (v-1) \alpha^* - v \alpha^* ) \beta \\
&\quad + M n (\alpha^* - 1 - M (1+\alpha^*)) \beta^2.
\end{aligned}
\end{equation}
\end{lemma}
\begin{proof}[Proof of Lemma~\ref{lem:lambda10}]
Combining~\eqref{eq.q_func} and~\eqref{eq.r_func}, we first define an affine function $s:\R\to\R$ as
\begin{equation}\label{eq.s_func}
\begin{aligned}
s(\beta) := \ &r(\beta) + \frac{q(\beta)}{(-1 + M)^2 n + (1 + M)^2 (-1 + n) v} \\
= \ &(2M\beta - (M+1))\cdot \frac{2(-n + (n-1)v)}{(-1 + M)^2 n + (1 + M)^2 (-1 + n) v} \cdot t(\alpha^*),
\end{aligned}
\end{equation}
where $t:\R\to\R$ is a quadratic function such as
\begin{equation}\label{eq.t_func}
    t(\alpha) := (-1 + M) M n \alpha^2 + \bigl(-M (1 + M) v + n (2 + M v + M^2 v)\bigr)\alpha - (1 + M) n.
\end{equation}
Our goal is to show that $r(\beta^*) \le 0$. Since $q(\beta^*) = 0$ (by~\eqref{eq.stepsize.beta} and~\eqref{eq.q_func}), from~\eqref{eq.s_func} and the fact of $(M,n,v)\in\R_{>1}\times\N\times\left[0, \tfrac{(M-1)n}{(n-1)M}\right]$ and $\beta^*\in\left(\frac{\alpha^*}{M}, \frac{M+1}{2M}\right)$ (by Theorem~\ref{lem:root-interval}), it is sufficient to prove $t(\alpha^*) \leq 0$. 

By~\eqref{eq.t_func} and $(M,n)\in\R_{>1}\times\N$, the quadratic function $t:\R\to\R$ is strongly convex and the equation $t(\alpha) = 0$ has exactly one positive solution, which is denoted as $\bar\alpha$ (i.e., $t(\bar\alpha) = 0$ and $\bar\alpha\in\R_{>0}$). By using the quadratic formula, we have
\begin{equation}\label{eq.bar_alpha}
\bar\alpha = \frac{M (1 + M) v - n (2 + M v + M^2 v) + \sqrt{(-M (1 + M) v + n (2 + M v + M^2 v))^2 + 4(-1 + M)(1+M) M n^2}}{2(-1 + M) M n}.
\end{equation}

Moreover, it follows~\eqref{eq.t_func} and $(M,n,v)\in\R_{>1}\times\N\times\left[0, \tfrac{(M-1)n}{(n-1)M}\right]$ that
\begin{equation}\label{eq.t_func_values}
\begin{aligned}
t(0) &= -(1+M)n < 0, \\
t(1) &= n(M - 1)^2 + (n-1)(1 + M) M v > 0, \\
\text{and}\quad t\left(\tfrac{1}{M}\right) &= \frac{M+1}{M}\Bigl( (1 - M)n + (n-1)M v \Bigr) \leq 0.
\end{aligned}
\end{equation}
Because $t:\R\to\R$ is a strongly convex quadratic function, it directly follows from~\eqref{eq.t_func_values} that 
\begin{equation}\label{eq.bar_alpha_property}
\bar{\alpha} \in \left[\tfrac{1}{M},1\right).
\end{equation}
To conclude $t(\alpha^*) \leq 0$, in the rest of this proof, we are going to show that $\alpha^* \in \left[\frac{1}{M},\bar\alpha\right)$. By Theorem~\ref{lem:root-interval}, $\alpha^*\in\left[\frac{1}{M},1\right)$ satisfies $h(\alpha^*) = 0$, where function $h:\R\to\R$ is defined in~\eqref{eq:h_func}. In addition, after combining~\eqref{eq:h_func} and~\eqref{eq.bar_alpha}, we have 
\begin{equation}\label{eq.lm32.subs_polished}
    h(\bar{\alpha}) = \frac{c_1\sqrt{c_2} + c_3}{2 (-1 + M)^2 M n^2},
\end{equation}
where $\{c_1, c_2, c_3\}\subset\R$ are 
\begin{align*}
    c_1 &= M (1 + M)^2 (n-1)^2 v^2 + (1 + 3 M - M^2 + M^3) (n-1) n v + (M-1)^2 n^2, \\
    c_2 &= M^2 (1 + M)^2 (n-1)^2 v^2 + 4 M (1 + M) (n-1) n v + 4 (1 - M + M^3) n^2, \\
    \text{and} \quad c_3 &= -M^2 (1 + M)^3 (n-1)^3 v^3 -M (1+M)(3+5M-M^2+M^3) (n-1)^2 n v^2 \\
    &\quad - (2 + 5 M - 5 M^2 + 3 M^3 + 3 M^4) (n-1) n^2 v - 2 (M-1)^2 (1 - M + M^2) n^3,
\end{align*}
respectively. Since $(M,n,v)\in\R_{>1}\times\N\times\left[0, \tfrac{(M-1)n}{(n-1)M}\right]$, every single term of $c_1$ and $c_2$ is non-negative, which implies $c_1\ge0$ and $c_2\ge0$. Meanwhile, from the same logic, every single term of $c_3$ is non-positive and thus $c_3\le0$. Moreover, because $h(\bar{\alpha})$ and $c_1^2 c_2 - c_3^2$ must be both non-positive or non-negative and a calculation process shows that for $(M,n,v)\in\R_{>1}\times\N\times\left[0, \tfrac{(M-1)n}{(n-1)M}\right]$,
\begin{align*}
    c_1^2c_2 - c_3^2 &= 4M(M-1)^2n^3\cdot\left(M(n-1)v - n(M-1)\right)\cdot\left((M+1)^2(n-1)v + (M-1)^2n\right)^2 \le 0,
\end{align*}
we conclude that $h(\bar\alpha) \leq 0$. From~\eqref{eq.bar_alpha_property}, $h(\bar\alpha) \leq 0$, $h\left(\tfrac{1}{M}\right) \geq 0$ (see~\eqref{eq:h_endpoint_val}), the existence of a unique $\alpha^*\in\left[\frac{1}{M},1\right)$ satisfying $h(\alpha^*) = 0$ (see Theorem~\ref{lem:root-interval}), and the intermediate value theorem, we know $\alpha^*\in \left[\tfrac{1}{M},\bar\alpha\right]$. Furthermore, it follows the strong convexity of function $t:\R\to\R$ (defined in~\eqref{eq.t_func}) and $\alpha^*\in \left[\tfrac{1}{M},\bar\alpha\right]\subset[0,\bar\alpha]$ that $t(\alpha^*) \leq \max\{t(0),t(\bar\alpha)\} = 0$. In the end, taking $\beta = \beta^*$ in~\eqref{eq.s_func} demonstrates that
\begin{align*}
r(\beta^*) = \ &r(\beta^*) + \frac{q(\beta^*)}{(-1 + M)^2 n + (1 + M)^2 (-1 + n) v} \\
= \ &(2M\beta^* - (M+1))\cdot \frac{2(-n + (n-1)v)}{(-1 + M)^2 n + (1 + M)^2 (-1 + n) v} \cdot t(\alpha^*) \leq 0,
\end{align*}
where the first equality is from $q(\beta^*) = 0$ (by~\eqref{eq.stepsize.beta} and~\eqref{eq.q_func}), the second equality follows~\eqref{eq.s_func}, and the inequality uses $\beta^* < \frac{M+1}{2M}$ (see Theorem~\ref{lem:root-interval}), $(M,n,v)\in\R_{>1}\times\N\times\left[0, \tfrac{(M-1)n}{(n-1)M}\right]$, and the fact of $t(\alpha^*)\leq 0$. Therefore, the proof is completed.
\end{proof}

The function $r(\beta)$ defined in~\eqref{eq.r_func} plays a key role in demonstrating the non-negativity of the multipliers $\lambda_{i,0}$ and $\lambda_{0,i}$ for all $i\in\Omega\backslash\{0,*\}$, as introduced in~\eqref{eq.lambda_0i}--\eqref{eq.lambda_i0}. In addition, our next lemma constructs a matrix with a carefully designed structure, which is then equating the left-hand-side matrix appearing in the SDP constraint~\eqref{eq.dual.sdpconstraint}, yielding a system of algebraic equations whose solution uniquely determines the dual variables as well as the two-step stepsize schedule $(\alpha^*,\beta^*)$. The structure of the constructed matrix is inspired by a generalization of the deterministic two-step silver stepsize schedule~\cite{altschuler2025acceleration}. In the deterministic setting (i.e., $n=1$), our structured matrix simplifies to an all-zeros matrix, thus recovering the classical two-step silver stepsize schedule. In contrast, in the stochastic setting (i.e., $n\in\mathbb{N}_{>1}$), the matrix depends on the choice of $v$, leading to distinct feasible dual variables and our proposed two-step stepsize schedules.

\begin{lemma}\label{lem:semidef}
With predetermined parameters $(M,n,v)\in\R_{>1}\times\N\times\left[0, \tfrac{(M-1)n}{(n-1)M}\right]$, for any $(z,\delta_1,\delta_2)\in\R_{>0}\times\R\times\R$, let's consider a symmetric matrix
\begin{equation}\label{eq.Lambda_def}
\Lambda = \begin{bmatrix} A & B \\ B^{\top} & C \end{bmatrix} \in\R^{(n+n^2)\times(n+n^2)},
\end{equation}
where $A = (n \delta_1 + (n-1)nvz)I_n - \delta_1J_n \in \R^{n\times n}$, $B\in\R^{n\times n^2}$ follows the structure that for any $(i,j)\in[n]\times[n^2]$, its $i$th-row-$j$th-column element can be represented as
\begin{equation*}
[B]_{(i,j)} = \begin{cases} (n-1)z &\text{if } i = \left\lceil \tfrac{j}{n} \right\rceil \\ -z &\text{otherwise,} \end{cases}
\end{equation*}
and $C = C_0^{\oplus n} \in\R^{n^2\times n^2}$ with $C_0 = (n\delta_2 + (n-1)vz)I_n - \delta_2J_n$, then the following statements hold true.
\begin{itemize}
\item[(a)] If $n=1$, then $\Lambda\in\R^{(n+n^2)\times(n+n^2)}$ defined in~\eqref{eq.Lambda_def} is positive semi-definite for any $(\delta_1,\delta_2)\in\R\times\R$.
\item[(b)] If $n\in\N_{>1}$, then $\Lambda\in\R^{(n+n^2)\times(n+n^2)}$ defined in~\eqref{eq.Lambda_def} is positive semi-definite if
\begin{equation}\label{ieq.sdp_cond}
\delta_1  \geq \frac{n^2z - (n-1)^2v^2z}{(n-1)v}\footnote{By convention, when $v=0$, we have $\frac{n^2z - (n-1)^2v^2z}{(n-1)v} = \frac{n^2z}{(n-1)v} = +\infty$ for any $(n,z)\in\N_{>1}\times\R_{>0}$.} \quad \text{and} \quad \delta_2 \geq -\frac{(n-1)}{n}vz.
\end{equation}
\end{itemize}
\end{lemma}
\begin{proof}[Proof of Lemma~\ref{lem:semidef}]
When $n=1$, from~\eqref{eq.Lambda_def} and the definitions of sub-matrices $(A,B,C)\in\R^{n\times n}\times\R^{n\times n^2}\times\R^{n^2\times n^2}$ described in the lemma statement, the matrix $\Lambda\in\R^{(n+n^2)\times(n+n^2)}$ is actually an all-zero matrix, i.e., $\Lambda = \mathbf{0}$. Therefore, for any $(\delta_1,\delta_2)\in\R\times\R$, the matrix $\Lambda$ is positive semi-definite. 

On the other hand, when $n\in\N_{>1}$, let's consider a matrix $W\in\R^{n\times (n-1)}$, whose columns consist of an orthonormal basis, satisfying 
\begin{equation}\label{eq.W_def}
W^\top W=I_{n-1} \quad \text{and} \quad W^\top \mathbf{1}_n = \mathbf{0},
\end{equation}
and we further define
\begin{equation}\label{eq.Xi_def}
\Xi := \begin{bmatrix} \tfrac{\mathbf{1}_n}{\sqrt{n}} & W \end{bmatrix} \in\R^{n\times n},
\end{equation}
where it holds that 
\begin{equation}\label{eq.Xi_orthogonal}
\Xi^\top\Xi = \begin{bmatrix} \tfrac{\mathbf{1}_n}{\sqrt{n}} & W \end{bmatrix}^\top \begin{bmatrix} \tfrac{\mathbf{1}_n}{\sqrt{n}} & W \end{bmatrix} = \begin{bmatrix} 1 & \mathbf{0} \\ \mathbf{0} & I_{n-1} \end{bmatrix} = I_n.
\end{equation}
Because $\Xi\in\R^{n\times n}$ is a square matrix, \eqref{eq.Xi_orthogonal} directly implies $\Xi^\top = \Xi^{-1}$ and $\Xi\Xi^\top = I_n$. Hence, $\Xi$ is an orthogonal matrix.

In the rest of this proof, for any $i\in[n^2]$, we denote $e_i\in\R^{n^2}$ as a vector whose $i$th element is one while all other elements are zero. Then we can define a permutation matrix $\Gamma\in\R^{n^2\times n^2}$ such that for any $i\in[n^2]$,
\begin{equation}\label{eq.Gamma_def}
\Gamma e_i = e_{\pi(i)} \quad \text{with} \quad \pi(i) = \begin{cases} (n-1)n + \left\lceil \tfrac{i}{n} \right\rceil &\text{if }\text{mod}(i-1,n) = 0 \\ \left\lfloor \tfrac{i-1}{n} \right\rfloor \cdot (n-1) + \text{mod}(i-1,n) &\text{if }\text{mod}(i-1,n) \neq 0.\end{cases}
\end{equation}
From~\eqref{eq.Gamma_def}, the permutation matrix $\Gamma\in\R^{n^2\times n^2}$ has the property that for any $i\in[n^2]$, its $\pi(i)$th-row-$i$th-column element equals to one while all other elements equal to zero. Meanwhile, because $\pi:[n^2]\to[n^2]$ is a bijection, the permutation matrix $\Gamma$ is an orthogonal matrix, i.e., $\Gamma^\top\Gamma = \Gamma\Gamma^\top = I_{n^2}$. 

Finally, let's define matrix $S\in\R^{(n+n^2)\times(n+n^2)}$ as
\begin{equation}\label{eq.S_def}
S := I_n \oplus (\Xi^{\oplus n}\Gamma^\top).
\end{equation}
By the orthogonality of matrices $\Xi\in\R^{n\times n}$ and $\Gamma\in\R^{n^2\times n^2}$, we further have 
\begin{equation*}
S^\top S = \begin{bmatrix} I_n & \mathbf{0} \\ \mathbf{0} & \Gamma(\Xi^{\top})^{\oplus n} \end{bmatrix} \begin{bmatrix} I_n & \mathbf{0} \\ \mathbf{0} & \Xi^{\oplus n}\Gamma^\top \end{bmatrix} = \begin{bmatrix} I_n & \mathbf{0} \\ \mathbf{0} & \Gamma(\Xi^\top\Xi)^{\oplus n}\Gamma^\top \end{bmatrix} = I_{n+n^2},
\end{equation*}
which demonstrates the orthogonality of $S$. Combining~\eqref{eq.Lambda_def} and~\eqref{eq.S_def}, we focus on analyzing the transformation matrix
\begin{equation}\label{eq.Lambda_transformation}
S^\top\Lambda S = \begin{bmatrix} I_n & \mathbf{0} \\ \mathbf{0} & \Gamma(\Xi^{\top})^{\oplus n} \end{bmatrix} \begin{bmatrix} A & B \\ B^\top & C \end{bmatrix} \begin{bmatrix} I_n & \mathbf{0} \\ \mathbf{0} & \Xi^{\oplus n}\Gamma^\top \end{bmatrix} = \begin{bmatrix} A & B\Xi^{\oplus n}\Gamma^\top \\ \Gamma(\Xi^\top)^{\oplus n}B^\top & \Gamma(\Xi^\top)^{\oplus n} C \Xi^{\oplus n}\Gamma^\top \end{bmatrix}.
\end{equation}
To start with, the $(2,2)$-block of the transformation matrix in~\eqref{eq.Lambda_transformation} is
\begin{equation}\label{eq.transformation_22}
\begin{aligned}
&\Gamma(\Xi^\top)^{\oplus n} C \Xi^{\oplus n}\Gamma^\top = \Gamma(\Xi^\top)^{\oplus n} C_0^{\oplus n} \Xi^{\oplus n}\Gamma^\top = \Gamma (\Xi^\top C_0 \Xi)^{\oplus n}\Gamma^\top \\
= \ &\Gamma \left(\Xi^\top ((n\delta_2 + (n-1)vz)I_n - \delta_2J_n) \Xi \right)^{\oplus n} \Gamma^\top = \Gamma \left( (n\delta_2 + (n-1)vz)\Xi^\top\Xi - \delta_2\Xi^\top J_n\Xi \right)^{\oplus n}\Gamma^\top \\
= \ &\Gamma \left( (n\delta_2 + (n-1)vz)I_n -  \begin{bmatrix} n\delta_2 & \mathbf{0} \\ \mathbf{0} & \mathbf{0} \end{bmatrix} \right)^{\oplus n}\Gamma^\top = \Gamma \left( \begin{bmatrix} (n-1)vz & \mathbf{0} \\ \mathbf{0} & (n\delta_2 + (n-1)vz)I_{n-1} \end{bmatrix} \right)^{\oplus n}\Gamma^\top \\
= \ &\begin{bmatrix} (n\delta_2 + (n-1)vz)I_{n(n-1)} & \mathbf{0} \\ \mathbf{0} & (n-1)vz I_n \end{bmatrix},
\end{aligned}
\end{equation}
where the first and third equalities follow definitions of matrices $C\in\R^{n^2\times n^2}$ and $C_0\in\R^{n\times n}$ in the lemma statement, the fourth equality follows the definition of $J_n$ and \eqref{eq.W_def}, the fifth equality uses~\eqref{eq.W_def}--\eqref{eq.Xi_orthogonal}, and the last equality relies on the definition of permutation matrix $\Gamma\in\R^{n^2\times n^2}$; see~\eqref{eq.Gamma_def}.

Then let's consider the $(1,2)$-block of the transformation matrix in~\eqref{eq.Lambda_transformation}, which is
\begin{equation}\label{eq.transformation_12}
\begin{aligned}
B\Xi^{\oplus n}\Gamma^\top = \ &\left((nz\mathbf{1}_n^\top)^{\oplus n} - z\mathbf{1}_{n\times n^2}\right)\Xi^{\oplus n}\Gamma^\top = \left((nz\mathbf{1}_n^\top\Xi)^{\oplus n} - zJ_n\underbrace{\begin{bmatrix} \Xi & \Xi & \cdots & \Xi \end{bmatrix}}_{n \text{ copies}}\right)\Gamma^\top \\
= \ &\left(\left(nz\begin{bmatrix} \sqrt{n} & \mathbf{0} \end{bmatrix}\right)^{\oplus n} - z\mathbf{1}_n\begin{bmatrix} \sqrt{n} & \mathbf{0} \end{bmatrix}\underbrace{\begin{bmatrix} I_n & I_n & \cdots & I_n \end{bmatrix}}_{n \text{ copies}}\right)\Gamma^\top \\
= \ &\left(\Gamma\left(\begin{bmatrix} nz\sqrt{n} \\ \mathbf{0} \end{bmatrix}\right)^{\oplus n} - \Gamma{\underbrace{\begin{bmatrix} I_n & I_n & \cdots & I_n \end{bmatrix}}_{n \text{ copies}}}^\top \begin{bmatrix} z\sqrt{n} \\ \mathbf{0} \end{bmatrix}\mathbf{1}_n^\top \right)^\top \\
= \ &\left( \begin{bmatrix} \mathbf{0} \\ nz\sqrt{n}I_n \end{bmatrix} - \begin{bmatrix} \mathbf{0} \\ z\sqrt{n} J_n \end{bmatrix} \right)^\top = \begin{bmatrix} \mathbf{0} & z\sqrt{n}(nI_n - J_n) \end{bmatrix},
\end{aligned}
\end{equation}
where the first equality follows the definition of $B\in\R^{n\times n^2}$ in the lemma statement, the third equality relies on~\eqref{eq.W_def}--\eqref{eq.Xi_orthogonal}, and the second last equality is from the definition of $\Gamma\in\R^{n^2\times n^2}$ in~\eqref{eq.Gamma_def}.

Combining~\eqref{eq.Lambda_def}, the orthogonality of $S\in\R^{(n+n^2)\times(n+n^2)}$ (defined in~\eqref{eq.S_def}), and~\eqref{eq.Lambda_transformation}--\eqref{eq.transformation_12}, we have matrix $\Lambda\in\R^{(n+n^2)\times(n+n^2)}$ shares same eigenvalues with
\begin{equation*}
S^\top\Lambda S = \begin{bmatrix} A & \mathbf{0} & z\sqrt{n}(nI_n - J_n) \\ \mathbf{0} & (n\delta_2 + (n-1)vz)I_{(n-1)n} & \mathbf{0} \\ z\sqrt{n}(nI_n - J_n) & \mathbf{0} & (n-1)vz I_n \end{bmatrix},
\end{equation*}
which further has the same eigenvalues as
\begin{align}
&\begin{bmatrix} (n\delta_2 + (n-1)vz)I_{(n-1)n} & \mathbf{0} & \mathbf{0} \\ \mathbf{0} & A & z\sqrt{n}(nI_n - J_n) \\ \mathbf{0} & z\sqrt{n}(nI_n - J_n) & (n-1)vz I_n \end{bmatrix} \nonumber{}  \\
& \quad \quad  = (n\delta_2 + (n-1)vz)I_{(n-1)n} \oplus \begin{bmatrix} A & z\sqrt{n}(nI_n - J_n) \\ z\sqrt{n}(nI_n - J_n) & (n-1)vz I_n \end{bmatrix}. \label{eq.eigenvalue_transformation}
\end{align}
Under the condition~\eqref{ieq.sdp_cond}, it is guaranteed that $(n\delta_2 + (n-1)vz)I_{(n-1)n} \succeq \mathbf{0}$. Therefore, from~\eqref{eq.eigenvalue_transformation}, we only need to show $\begin{bmatrix} A & z\sqrt{n}(nI_n - J_n) \\ z\sqrt{n}(nI_n - J_n) & (n-1)vz I_n \end{bmatrix} \succeq \mathbf{0}$ to conclude the statement. It follows the orthogonality of matrix $\Xi\in\R^{n\times n}$ (see~\eqref{eq.Xi_def}--\eqref{eq.Xi_orthogonal}) that $\begin{bmatrix} A & z\sqrt{n}(nI_n - J_n) \\ z\sqrt{n}(nI_n - J_n) & (n-1)vz I_n \end{bmatrix}$ shares the same eigenvalues as
\begin{align}
&\begin{bmatrix} \Xi^\top & \mathbf{0} \\ \mathbf{0} & \Xi^\top \end{bmatrix}
\begin{bmatrix} A & z\sqrt{n}(nI_n - J_n) \\ z\sqrt{n}(nI_n - J_n) & (n-1)vz I_n \end{bmatrix}\begin{bmatrix} \Xi & \mathbf{0} \\ \mathbf{0} & \Xi \end{bmatrix} \nonumber{} \\
= \ &\begin{bmatrix} \Xi^\top A\Xi & z\sqrt{n}(n\Xi^\top I_n\Xi - \Xi^\top J_n\Xi) \\ z\sqrt{n}(n\Xi^\top I_n\Xi - \Xi^\top J_n\Xi) & (n-1)vz\Xi^\top I_n\Xi \end{bmatrix} \nonumber{} \\
= \ &\begin{bmatrix} (n \delta_1 + (n-1)nvz)\Xi^\top\Xi - \delta_1\Xi^\top\mathbf{1}_n\mathbf{1}_n^\top\Xi  & zn\sqrt{n}\Xi^\top\Xi - z\sqrt{n} \Xi^\top \mathbf{1}_n\mathbf{1}_n^\top\Xi \\ zn\sqrt{n}\Xi^\top\Xi - z\sqrt{n} \Xi^\top \mathbf{1}_n\mathbf{1}_n^\top\Xi & (n-1)vz\Xi^\top\Xi \end{bmatrix} \nonumber{} \\
= \ &\begin{bmatrix} (n-1)nvz \oplus (n \delta_1 + (n-1)nvz)I_{n-1} & 0 \oplus zn\sqrt{n}I_{n-1} \\ 0 \oplus zn\sqrt{n}I_{n-1} & (n-1)vzI_n  \end{bmatrix}, \label{eq.transformation_matrix_end}
\end{align}
where the second equality follows the definition of matrix $A\in\R^{n\times n}$ in the original statement and $J_n = \mathbf{1}_n\mathbf{1}_n^\top$, and the last equality uses~\eqref{eq.W_def}--\eqref{eq.Xi_orthogonal}. Moreover, due to the structure of the matrix in~\eqref{eq.transformation_matrix_end}, $\begin{bmatrix} A & z\sqrt{n}(nI_n - J_n) \\ z\sqrt{n}(nI_n - J_n) & (n-1)vz I_n \end{bmatrix}$ further shares the same eigenvalues as the matrix
\begin{equation}\label{eq.matrix_eigenvalue_equi}
\begin{bmatrix}
    (n-1)nvz & 0 \\
    0 & (n-1)vz
\end{bmatrix} \oplus \left(\begin{bmatrix} n\delta_1 + (n-1)nvz & zn\sqrt{n} \\ zn\sqrt{n} & (n-1)vz \end{bmatrix}\right)^{\oplus (n-1)}.
\end{equation}
With $(n,v,z)\in\N_{>1}\times\left[0, \tfrac{(M-1)n}{(n-1)M}\right]\times\R_{>0}$, the matrix in~\eqref{eq.matrix_eigenvalue_equi} is positive semi-definite when 
\begin{equation*}
\begin{bmatrix} n\delta_1 + (n-1)nvz & zn\sqrt{n} \\ zn\sqrt{n} & (n-1)vz \end{bmatrix}\succeq \mathbf{0},
\end{equation*}
which happens if and only if
\begin{equation}\label{eq.psd_last_condition}
n\delta_1 + (n-1)nvz \geq 0 \quad\text{and}\quad (n\delta_1 + (n-1)nvz)(n-1)vz - z^2n^3 \geq 0.
\end{equation}
Under the condition~\eqref{ieq.sdp_cond}, \eqref{eq.psd_last_condition} always holds. Therefore, when $n\in\N_{>1}$, \eqref{ieq.sdp_cond} is sufficient for $\Lambda\in\R^{(n+n^2)\times(n+n^2)}$ defined in~\eqref{eq.Lambda_def} to be positive semi-definite. The proof is completed.
\end{proof}

Now we are ready to present the proof of Theorem~\ref{theo:dual_feasibility}.

\begin{proof}[Proof of Theorem~\ref{theo:dual_feasibility}]
Under the parameter setting described in the theorem statement, we prove statements~(i)--(iii) as follows.

\noindent \textbf{Statement (i).} By Theorem~\ref{lem:root-interval}, $\alpha^*\in\left[\frac{1}{M},1\right)$ and $\beta^*\in\left(\frac{\alpha^*}{M},\frac{M+1}{2M}\right)\subseteq\left(\frac{1}{M^2},\frac{M+1}{2M}\right)$, which ensures that for any $i\in\Omega\backslash\{0,*\}$, $\lambda_{i,*}$ and $\lambda_{*,i}$ defined in~\eqref{eq.lambda_i*}--\eqref{eq.lambda_*i} satisfy
\begin{align*}
\lambda_{i,*} &= \frac{\alpha^* + \beta^* - \alpha^*\beta^*}{n(1+M-\alpha^*+M\alpha^*)} = \frac{-(\alpha^* - 1)(\beta^* - 1) + 1}{n(1+M+(M-1)\alpha^*)} > 0 \\
\text{and} \quad \lambda_{*,i} &= \frac{M\beta^* + \alpha^*(-1 + M\beta^*)}{Mn(1+M-\alpha^*+M\alpha^*)} = \frac{(M\beta^* - \alpha^*) + M\alpha^*\beta^*}{Mn(1+M+(M-1)\alpha^*)} > 0.
\end{align*}
Meanwhile, $\lambda_{0,*}$ is non-negative by~\eqref{eq.lambda_0*}. For $\lambda_{*,0}$ in~\eqref{eq.lambda_*0}, after substituting the expressions of $\lambda_{1^{(1)},*}$ and $\lambda_{*,1^{(1)}}$ from~\eqref{eq.lambda_i*}--\eqref{eq.lambda_*i}, we have 
\begin{equation}\label{eq.lambda_*0_pos}
\begin{aligned}
\lambda_{*,0} = n(\lambda_{1^{(1)},*} - \lambda_{*,1^{(1)}}) &= n\left(\frac{\alpha^* + \beta^* - \alpha^*\beta^*}{n(1+M-\alpha^*+M\alpha^*)} - \frac{M\beta^* + \alpha^*(-1 + M\beta^*)}{Mn(1+M-\alpha^*+M\alpha^*)}\right) \\
&= \frac{\alpha^*(1+M-2M\beta^*)}{M(1+M + (M-1)\alpha^*)} > 0,
\end{aligned}
\end{equation}
where the last inequality comes from Theorem~\ref{lem:root-interval} that $\alpha^* > 0$ and $\beta^* < \frac{M+1}{2M}$.

Next, let's show that, following the definition in~\eqref{eq.lambda_i0}, $\lambda_{i,0} \geq 0$ for any $i\in\Omega\backslash\{0,*\}$. For any $i\in\Omega\backslash\{0,*\}$, by~\eqref{eq.lambda_i*} and~\eqref{eq.lambda_*i}, the expression of $\lambda_{i,0}$ in~\eqref{eq.lambda_i0} turns to
\begin{align}
\lambda_{i,0} &= \frac{n^2(\lambda_{1^{(1)},*}+M\lambda_{*,1^{(1)}})^2 - n(\lambda_{1^{(1)},*}+\lambda_{*,1^{(1)}})}{2(n - nv + v)} - \frac{\lambda_{1^{(1)},*} - \lambda_{*,1^{(1)}}}{2} \nonumber{} \\
&= \frac{n^2\left(\frac{\alpha^* + \beta^* - \alpha^*\beta^*}{n(1+M-\alpha^*+M\alpha^*)}+M\left(\frac{M\beta^* + \alpha^*(-1 + M\beta^*)}{Mn(1+M-\alpha^*+M\alpha^*)}\right)\right)^2 - n\left(\frac{\alpha^* + \beta^* - \alpha^*\beta^*}{n(1+M-\alpha^*+M\alpha^*)}+\frac{M\beta^* + \alpha^*(-1 + M\beta^*)}{Mn(1+M-\alpha^*+M\alpha^*)}\right)}{2(n - nv + v)} \nonumber{} \\
&\quad\quad\quad - \frac{\frac{\alpha^* + \beta^* - \alpha^*\beta^*}{n(1+M-\alpha^*+M\alpha^*)} - \frac{M\beta^* + \alpha^*(-1 + M\beta^*)}{Mn(1+M-\alpha^*+M\alpha^*)}}{2} \nonumber{} \\
&= \frac{-r(\beta^*)}{2Mn(n - (n-1)v)(1+M + (M-1)\alpha^*)} \geq 0, \label{eq.lambda_i0_pos}
\end{align}
where the last equality uses the definition of function $r:\R\to\R$ in~\eqref{eq.r_func}, and the last inequality follows Lemma~\ref{lem:lambda10} and $(M,n,v)\in\R_{>1}\times\N\times\left[0, \tfrac{(M-1)n}{(n-1)M}\right]$.

By~\eqref{eq.lambda_*0}--\eqref{eq.lambda_i0}, for any $i\in\Omega\backslash\{0,*\}$,
\begin{align}\label{eq.lambda_0i_positive}
\lambda_{0,i} - \lambda_{i,0} = \lambda_{1^{(1)},*} - \lambda_{*,1^{(1)}} = \frac{\lambda_{*,0}}{n} > 0,
\end{align}
where the last inequality follows~\eqref{eq.lambda_*0_pos}. From the fact of $\lambda_{i,0} \geq 0$ for all $i\in\Omega\backslash\{0,*\}$ (see~\eqref{eq.lambda_i0_pos}), we further conclude $\lambda_{0,i} > 0$ for any $i\in\Omega\backslash\{0,*\}$.

Moreover, from~\eqref{eq.lambda_i*}--\eqref{eq.lambda_*0} and~\eqref{eq.tau}, it holds that
\begin{align*}
\tau &= 1 - M\left(\lambda_{0,*} + \lambda_{*,0} + \sum_{i=1}^n\lambda_{1^{(i)},*} + \sum_{i=1}^n\lambda_{*,1^{(i)}}\right) = 1 - M\left( n(\lambda_{1^{(1)},*} - \lambda_{*,1^{(1)}}) + n\lambda_{1^{(1)},*} + n\lambda_{*,1^{(1)}} \right) \\
&= 1 - 2Mn\lambda_{1^{(1)},*} = 1 - \frac{2M(\alpha^* + \beta^* - \alpha^*\beta^*)}{1+M-\alpha^*+M\alpha^*} = \frac{(1-\alpha^*)(1+M-2M\beta^*)}{1+M + (M-1)\alpha^*} > 0,
\end{align*}
where the last inequality comes from $M\in\R_{>1}$, $\alpha^*\in\left[\frac{1}{M},1\right)$, and $\beta^* \in \left(\frac{\alpha^*}{M},\frac{M+1}{2M}\right)$ by Theorem~\ref{lem:root-interval}.

Finally, $\mu_0$ and $\{\mu_i\}_{i\in[n]}$ are all non-negative by their definitions in~\eqref{eq.mu_0}--\eqref{eq.mu_i}. This completes the proof of statement (i).

\noindent \textbf{Statement (ii).} To show that the selected variables following~\eqref{eq.lambda_i*}--\eqref{eq.mu_i} satisfy the positive semi-definiteness constraint and the equality constraint in~\eqref{prob.dual.sdp}, we compute the entries of the left-hand-side matrix in~\eqref{eq.dual.sdpconstraint} 
\begin{align}\label{eq.Delta_def}
\Delta := \tau \bar{A}_R + \sum_{i=0}^n \mu_i \bar{A}_i^{\text{var}} - \sum\limits_{(i, j)\in\Omega\times\Omega \text{ with }i\neq j} \lambda_{i,j}\bar{B}_{i,j} - \bar{C} \in\R_{(1+n+n^2)\times(1+n+n^2)}.
\end{align}
By~\eqref{eq.lambda_i*},~\eqref{eq.lambda_*i},~\eqref{eq.lambda_0i}, and~\eqref{eq.lambda_i0}, for any $i\in\Omega\backslash\{0,*\}$,
\begin{align*}
\lambda_{i,*} = \lambda_{1^{(1)},*}, \quad \lambda_{*,i} = \lambda_{*,1^{(1)}}, \quad \lambda_{0,i} = \lambda_{0,1^{(1)}}, \quad \text{and} \quad \lambda_{i,0} = \lambda_{1^{(1)},0}.
\end{align*}
Therefore, in the following formulation, we use $(\lambda_{1^{(1)},*}, \lambda_{*,1^{(1)}}, \lambda_{0,1^{(1)}}, \lambda_{1^{(1)},0})$ to represent other equivalent terms for brevity. 

It's worth mentioning that all symbolic matrices in~\eqref{eq:Abar}--\eqref{eq:AR} are symmetric by their definitions, which leads to a symmetric $\Delta$ matrix by~\eqref{eq.Delta_def}. After representing symbolic matrices in~\eqref{eq:Abar}--\eqref{eq:AR} by symbolic vectors defined in~\eqref{eq.sdp.symobvec}--\eqref{eq.sdp.symobvec_1}, we can further quantify the $\Delta\in\R_{(1+n+n^2)\times(1+n+n^2)}$ matrix in~\eqref{eq.Delta_def}. It turns out the $\Delta$ matrix includes four structured pieces described in Figure~\ref{fig:matrix-structure}, whose entries are presented as follows \footnote{The detailed format of $\Delta$ can also be verified using \textit{Mathematica}; the corresponding implementation is available at \url{https://github.com/LuweiYY/LongStepSize_SGD}.}.

\begin{figure}[t]
    \centering
    \includegraphics[width=0.45\textwidth]{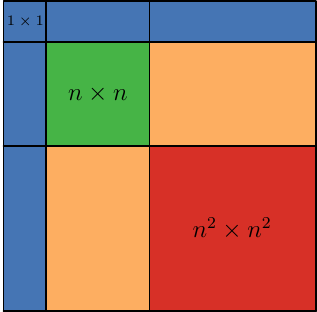}
    \caption{Structured $\Delta\in\R_{(1+n+n^2)\times(1+n+n^2)}$ matrix. }
    \label{fig:matrix-structure}
\end{figure}
\paragraph{First row and first column of $\Delta$ (blue area in Figure~\ref{fig:matrix-structure}).}
The top-left entry of $\Delta$ is given by:
\begin{align*}
    [\Delta]_{(1,1)} &= \tau - 1 + M\left(\lambda_{0,*} + \lambda_{*,0} + n\lambda_{1^{(1)},*} + n\lambda_{*,1^{(1)}}\right).
\end{align*}
The rest of entries in the first row of $\Delta$ consist of two distinct constant values:
\begin{align*}
    [\Delta]_{(1, i)} = \begin{cases} \frac{1}{n}\left(\alpha - \lambda_{0,*} - M\lambda_{*,0} - n M\alpha\left(\lambda_{1^{(1)},*} + \lambda_{*,1^{(1)}}\right)\right) & \text{if } i \in \{2, \dots, n+1\}, \\
    \frac{1}{n^2}\left(\beta - n \lambda_{1^{(1)},*} - n M\lambda_{*,1^{(1)}}\right) &\text{if } i \in \{n+2, \dots, 1+n+n^2\}.
    \end{cases}
\end{align*}

\paragraph{Top-left $n \times n$ block of $\Delta$ (green area in Figure~\ref{fig:matrix-structure}).}
This sub-matrix only includes entries taking two different values, depending on diagonal or off-diagonal: for any $(i,j)\in\{2,\ldots,n+1\}\times\{2,\ldots,n+1\}$,
\begin{align*}
[\Delta]_{(i,j)} = \begin{cases}
\frac{1}{n^2}\Bigl((n-1)\mu_0 + n(\lambda_{0,1^{(1)}} + \lambda_{1^{(1)},0}) - 2n\alpha(\lambda_{0,1^{(1)}} + M\lambda_{1^{(1)},0}) \\
\quad + \lambda_{0,*} + \lambda_{*,0} + n\alpha^2\bigl(-1 + nM(\lambda_{0,1^{(1)}} + \lambda_{1^{(1)},0} + \lambda_{*,1^{(1)}} + \lambda_{1^{(1)},*})\bigr) \Bigr) & \text{if }i=j, \\
\frac{1}{n^2}\left(-\mu_0 + n(\lambda_{0, 1^{(1)}} + \lambda_{1^{(1)}, 0}) -2n\alpha(\lambda_{0,1^{(1)}} + M\lambda_{1^{(1)},0}) + (\lambda_{0,*} + \lambda_{*,0})\right) & \text{if }i\neq j.
\end{cases}
\end{align*}

\paragraph{Bottom-right $n^2 \times n^2$ block of $\Delta$ (red area in Figure~\ref{fig:matrix-structure}).}
This area is a block-diagonal, square matrix. In particular, its diagonal blocks consist of $n$ sub-matrices with dimension of $n$-by-$n$. Each diagonal sub-matrix has entries taking two different values, depending on diagonal elements or not: for any $(i,j)\in\{n+2,\ldots,1+n+n^2\}\times\{n+2,\ldots,1+n+n^2\}$,
\begin{align*}
[\Delta]_{(i,j)} = \begin{cases}
0 & \text{if }\lceil \tfrac{i-1}{n} \rceil \neq \lceil \tfrac{j-1}{n} \rceil, \\
\frac{1}{n^2} \left(-\beta^2 + (n-1)\mu_1 + \lambda_{0,1^{(1)}} + \lambda_{1^{(1)},0} + \lambda_{1^{(1)},*} + \lambda_{*,1^{(1)}}\right) & \text{if }i=j, \\
\frac{1}{n^2} \left(-\mu_1 + \lambda_{0,1^{(1)}} + \lambda_{1^{(1)},0} + \lambda_{1^{(1)},*} + \lambda_{*,1^{(1)}}\right) & \text{otherwise.}
\end{cases}
\end{align*}

\paragraph{The remaining two blocks in $\Delta$ (orange area in Figure~\ref{fig:matrix-structure}).}
These two blocks also include entries that only take one of two possible values: for any $(i,j)\in\{2,\ldots,n+1\}\times\{n+2,\ldots,1+n+n^2\}\cup\{n+2,\ldots,1+n+n^2\}\times\{2,\ldots,n+1\}$,
\begin{align*}
[\Delta]_{(i,j)} = \begin{cases}
\frac{1}{n^2}\left(-\alpha\beta - (\lambda_{0,1^{(1)}} + \lambda_{1^{(1)},0}) + n\alpha(\lambda_{1^{(1)},0} + \lambda_{1^{(1)},*} + M (\lambda_{0,1^{(1)}} + \lambda_{*,1^{(1)}})) \right) &\text{if }i < j \text{ and } i = \lceil \frac{j-1}{n} \rceil, \\
\frac{1}{n^2}(-\lambda_{0,1^{(1)}}- \lambda_{1^{(1)},0}) &\text{if }i < j \text{ and } i \neq \lceil \frac{j-1}{n} \rceil, \\
[\Delta]_{(j,i)} &\text{if }i > j.
\end{cases}
\end{align*}

After identifying the detailed format of the matrix $\Delta\in\R_{(1+n+n^2)\times(1+n+n^2)}$, we turn to focus on the two-step stepsize schedule $(\alpha^*,\beta^*)$ uniquely defined in Theorem~\ref{lem:root-interval} together with the dual multipliers in~\eqref{eq.lambda_i*}--\eqref{eq.mu_i}, and present their properties. Consider following equations. In fact, by the definition of $(\alpha^*,\beta^*)$ in Theorem~\ref{lem:root-interval} and the dual multipliers in~\eqref{eq.lambda_i*}--\eqref{eq.mu_i}, we directly have the following equation system holds true \footnote{This system of equations can be verified via \textit{Mathematica}; see \url{https://github.com/LuweiYY/LongStepSize_SGD}.}:
\begin{align}\label{eq.equation_system}
\begin{cases}
\tau - 1 = -M\left(\lambda_{0,*} + \lambda_{*,0} + n\lambda_{1^{(1)},*} + n\lambda_{*,1^{(1)}}\right), \\
\alpha^* = \lambda_{0,*} + M\lambda_{*,0} + nM\alpha^*\left(\lambda_{1^{(1)},*} + \lambda_{*,1^{(1)}}\right), \\
\beta^* = n\lambda_{1^{(1)},*} + nM\lambda_{*,1^{(1)}},  \\
(\alpha^*)^2 = (n - (n-1)v)(\lambda_{0,1^{(1)}} + \lambda_{1^{(1)},0}) + (\lambda_{0,*} + \lambda_{*,0}) \\
\qquad\qquad + nM(\alpha^*)^2(\lambda_{0,1^{(1)}} + \lambda_{1^{(1)},0} + \lambda_{1^{(1)},*} + \lambda_{*,1^{(1)}}) - 2n\alpha^*\lambda_{0,1^{(1)}} - 2nM\alpha^*\lambda_{1^{(1)},0},  \\
(\beta^*)^2 = n\lambda_{*,1^{(1)}} + n\lambda_{1^{(1)},*} + (n - (n-1)v)( \lambda_{0,1^{(1)}} + \lambda_{1^{(1)},0}), \\
\alpha^*\beta^* = -n(\lambda_{0,1^{(1)}} + \lambda_{1^{(1)},0}) + nM\alpha^*(\lambda_{0,1^{(1)}} + \lambda_{*,1^{(1)}}) + n\alpha^*(\lambda_{1^{(1)},0} + \lambda_{1^{(1)},*}), \\
n\lambda_{0,1^{(1)}}+\lambda_{0,*} = n\lambda_{1^{(1)}, 0}+\lambda_{*,0},  \\
\lambda_{1^{(1)}, 0} + \lambda_{1^{(1)}, *} = \lambda_{0, 1^{(1)}} + \lambda_{*, 1^{(1)}},  \\
\lambda_{*, 0} + n\lambda_{*, 1^{(1)}} = \lambda_{0, *} + n\lambda_{1^{(1)}, *}. 
\end{cases}
\end{align}
 
Substituting the system of equations~\eqref{eq.equation_system} into the explicit entries of $\Delta$ under the condition of $(\alpha,\beta) = (\alpha^*,\beta^*)$ leads to simplifications. In fact, $\Delta$ directly collapses into the form of $\begin{bmatrix} 0 & \mathbf{0} \\ \mathbf{0} & \Lambda \end{bmatrix}$, where $\Lambda$ is defined in~\eqref{eq.Lambda_def} with
\begin{equation}\label{eq.z_delta_values}
\begin{aligned}
    z &= \frac{1}{n^2}\left(\lambda_{0,1^{(1)}} + \lambda_{1^{(1)},0}\right),\\
    \delta_1 &= \frac{1}{n^2}\Bigl( \mu_0 - n(\lambda_{0,1^{(1)}} + \lambda_{1^{(1)},0}) + 2n\alpha^*(\lambda_{0,1^{(1)}} + M\lambda_{1^{(1)},0}) - \lambda_{0,*} - \lambda_{*,0} \Bigr) , \\
    \text{and} \quad \delta_2 &= \frac{1}{n^2}\Bigl(\mu_1 - \lambda_{0,1^{(1)}} - \lambda_{1^{(1)},0} - \lambda_{1^{(1)},*} - \lambda_{*,1^{(1)}} \Bigr).
\end{aligned}
\end{equation}
According to the values of dual multipliers in~\eqref{eq.lambda_i*}--\eqref{eq.mu_i}, when $n \in \N_{>1}$, $(z,\delta_1,\delta_2)$ in~\eqref{eq.z_delta_values} satisfy
\begin{align*}
    z = \frac{1}{n^2}\left(\lambda_{0,1^{(1)}} + \lambda_{1^{(1)},0}\right) > 0, \qquad
    \delta_1 &\geq \frac{n^2z - (n-1)^2v^2z}{(n-1)v}, \qquad
    \text{and} \quad \delta_2 \geq -\frac{(n-1)}{n} vz,
\end{align*}
where the first inequality follows~\eqref{eq.lambda_0i_positive} and the result of non-negative dual multipliers already shown in statement (i), and the last two inequalities are directly from~\eqref{eq.z_delta_values} together with the definitions of $\mu_0$ and $\mu_1$ in~\eqref{eq.mu_0} and~\eqref{eq.mu_i}, respectively. Hence, by Lemma~\ref{lem:semidef}, the left-hand-side matrix in~\eqref{eq.dual.sdpconstraint} is positive semi-definite under the theorem statement. Lastly, the equality constraint in~\eqref{prob.dual.sdp} is naturally satisfied following the definition of dual multipliers in~\eqref{eq.lambda_i*}--\eqref{eq.lambda_ij} and the last three equalities in~\eqref{eq.equation_system}. This completes the proof of statement (ii).

\noindent \textbf{Statement (iii).} From statements (i) and (ii), the constructed dual variables in~\eqref{eq.lambda_i*}--\eqref{eq.mu_i} are feasible to the dual problem~\eqref{prob.dual.sdp}. 
As noted in Remark~\ref{rem:prime-dual-equal}, the objective value of any feasible solution to dual problem~\eqref{prob.dual.sdp} performs as a valid upper bound for the optimal objective value of primal problem~\eqref{pep:primal_prob}, which is equivalent to problem~\eqref{pep:one-step}. Therefore, \eqref{eq.upperbound} holds true, which conclude the statement.
\end{proof}

\subsection{Proof of Theorem~\ref{theo:lowerbound}}\label{sec.app_proof_3}
\begin{proof}
The proof consists of two steps. We begin with constructing a feasible solution to the primal PEP problem~\eqref{pep:one-step} under the constant stepsize schedule $\alpha=\beta=\tfrac{2}{M+1}$, resulting in a lower bound for $h_{\mathrm{constant}}(R,\sigma)$.
Next, we demonstrate the existence of parameters $\mathscr{U}(M, n) > 0$ and $\bar v \in \left[0, \tfrac{(M-1)n}{(n-1)M}\right]$ such that $\tfrac{\sigma}{R} \leq \sqrt{\mathscr{U}(M, n)}$ implies $\frac{h_{\mathrm{silver}}(\bar v,R,\sigma)}{h_{\mathrm{constant}}(R,\sigma)} \leq \mathscr{C} < 1$, which concludes the statement.

Let's start with considering the $d$-dimensional quadratic function
\begin{align*}
    f(x) \;=\; \frac{M}{2} ([x]_1)^2 + \frac{1}{2}\sum_{i=2}^d ([x]_i)^2,
\end{align*}
which belongs to the function class $\mathcal{F}_{1,M}$ and admits a unique minimizer at $x_* = \mathbf{0}$.
Let $e_1$ denote the first column of identity matrix $I_d$.
We set the initial iterate to $x_0 = R e_1$ so that $\|x_0\| = R$ and the initialization condition~\eqref{eq:initial_condition} is satisfied. Let's define $\boldsymbol{\sigma}_d := \sigma e_1\in\mathbb{R}^d$ so that $\|\boldsymbol{\sigma}_d\| = \sigma$.
We consider the stochastic gradient sampling policy as follows:
\begin{itemize}
\item When $n$ is an even number, for any $x\in\mathbb{R}^d$, there are $\frac{n}{2}$ samples taking the gradient estimates of $\nabla f(x) + \boldsymbol{\sigma}_d$ and the other
$\frac{n}{2}$ samples taking the gradient estimates of $\nabla f(x) - \boldsymbol{\sigma}_d$.
\item When $n$ is an odd number, for any $x\in\mathbb{R}^d$, there is only one sample achieving exact gradient information $\nabla f(x)$, while there are $\frac{n-1}{2}$ samples taking the gradient estimates of $\nabla f(x) + \sqrt{\frac{n}{n-1}}\,\boldsymbol{\sigma}_d$ and the rest of $\frac{n-1}{2}$ samples taking estimates of $\nabla f(x) - \sqrt{\frac{n}{n-1}}\,\boldsymbol{\sigma}_d$.
\end{itemize}
This construction ensures that the stochastic gradient estimates are unbiased estimators with bounded variance, satisfying conditions~\eqref{eq:expected_val_condition} and~\eqref{eq:variance_condition}.

Under the constant stepsize choice $\alpha = \beta = \tfrac{2}{M + 1}$, 
the iterates remain in the linear span of $\{e_1\}$ and the one-dimensional calculation applies. In particular, after two iterations, the expected squared distance to the optimal solution can be computed explicitly as
\begin{align}\label{eq.feasbile_perf}
    \mathbb{E}_{\xi_0, \xi_1}\!\left[\|x_2\|^2\right] 
    = \left(\frac{M - 1}{M + 1}\right)^{\!4} R^2 
      + \frac{8(1 + M^2)}{(1 + M)^4}\,\sigma^2.
\end{align}
Therefore, by the feasibility of our constructed problem to~\eqref{pep:one-step} and the definition of $h_{\rm constant}(R, \sigma)$ from the theorem statement, we immediately from~\eqref{eq.feasbile_perf} have that
\begin{align}\label{eq.h1hconstant}
    h_1(\tfrac{\sigma}{R}) := \bar{\tau} + \bar{\mu}\,\left(\frac{\sigma}{R}\right)^2 \le \frac{h_{\rm constant}(R, \sigma)}{R^2},
\end{align}
where 
\begin{align}\label{eq.bar_tau_mu}
    \bar{\tau} := \left(\frac{M - 1}{M + 1}\right)^{\!4} \quad \text{and} \quad \bar{\mu} := \frac{8(1 + M^2)}{(1 + M)^4}.    
\end{align}
Using $h(v,\tfrac{\sigma}{R})$ defined in~\eqref{eq.objfunc}, which, according to~\eqref{eq.upperbound}, behaves an upper bound of the optimal objective value of~\eqref{pep:one-step}. Thus, for any $v\in\left[0, \tfrac{(M-1)n}{(n-1)M}\right]$ we always have 
\begin{align}\label{eq.h2hours}
h(v, \tfrac{\sigma}{R}) \geq \frac{h_{\rm silver}(v, R, \sigma)}{R^2}.
\end{align}
As $(\alpha^{*}(v),\beta^{*}(v))$ is the unique solution satisfying~\eqref{eq.stepsize.range}--\eqref{eq.stepsize.beta} (see Theorem~\ref{lem:root-interval}) being continuous in $v$, $\tau(v)$ defined in~\eqref{eq.tau} and each $\{\mu_i(v)\}_{i\in\{0\}\cup[n]}$ defined in~\eqref{eq.mu_0}--\eqref{eq.mu_i} are continuous in $v$ as well. Consequently, $\mu(v)=\sum_{i=0}^{n}\mu_i(v)$ defined in~\eqref{eq.objfunc} is continuous as well.
Using closed-form expressions of $\tau(0)$ and $\bar\tau$ in~\eqref{eq.tauSilver} and~\eqref{eq.bar_tau_mu}, we have
\begin{align*}
    \tau(0) = \left(\frac{\sqrt{2M^2 - 2M + 1} - M}{2 - M + \sqrt{2M^2 - 2M + 1}}\right)^2 < \left(\frac{M - 1}{M + 1}\right)^{\!4} = \bar\tau.
\end{align*}
Together with the definition of $\mathscr{C}$ in~\eqref{eq.ratio}, we further know that
\begin{align*}
\tau(0) - \mathscr{C}\,\bar\tau = \tau(0) - \Bigl(\tfrac{1}{2} + \tfrac{\tau(0)}{2\bar\tau}\Bigr)\bar\tau = \tfrac{1}{2}\bigl(\tau(0) - \bar\tau\bigr) < 0.
\end{align*}
By the continuity of $\tau(v)$, there exists $\delta_1 \in \left(0, \tfrac{(M-1)n}{(n-1)M}\right]$ such that
\begin{align}\label{eq.numer.leq0}
    \tau(v) - \mathscr{C}\bar{\tau} < 0 \quad \text{for any } v\in [0,\delta_1].
\end{align}
Since $\mu(0) = +\infty$ from Remark~\ref{rem:dual_feasibility}, it holds that $\mathscr{C}\bar{\mu} - \mu(0) < 0$. Therefore, by the continuity of $\mu(v)$, there exists $\delta_2 \in \left(0, \tfrac{(M-1)n}{(n-1)M}\right]$ such that 
\begin{align}\label{eq.denom.leq0}
    \mathscr{C}\bar{\mu} - \mu(v) < 0 \quad \text{for any } v\in [0,\delta_2].
\end{align}
By~\eqref{eq.numer.leq0} and~\eqref{eq.denom.leq0}, it yields that
\begin{align}\label{eq.upbound.sigma}
    \frac{\tau(\tilde v)-\mathscr{C}\bar\tau}
         {\mathscr{C}\bar\mu-\mu(\tilde v)}>0 \quad \text{for any } \tilde v\in\left(0,\min\{\delta_1,\delta_2\}\right).
\end{align}
Let's define
\begin{align*}
\mathscr{U}(M,n)
:= \max_{v\in(0,v_{\max})}
\frac{\tau(v)-\mathscr{C}\bar{\tau}}{\mathscr{C}\bar{\mu}-\mu(v)}, \ \ \text{where}\ v_{\max} := \sup\left\{v\in \left[0, \tfrac{(M-1)n}{(n-1)M}\right]:\mathscr{C}\bar\mu-\mu(t)< 0 \ \text{for all} \ t \in [0,v] \,\right\}.
\end{align*}
By~\eqref{eq.denom.leq0}, we have $v_{\max} > 0$. Moreover, \eqref{eq.upbound.sigma} guarantees $\mathscr{U}(M,n)>0$.
Since $\mathscr{C}\bar{\mu}-\mu(v) < 0$ for any $v\in(0,v_{\max})$ and $(\tau(v),\mu(v))$ are both continuous on the interval $(0,v_{\max})$, we have that $\frac{\tau(v)-\mathscr{C}\bar{\tau}}{\mathscr{C}\bar{\mu}-\mu(v)}$ is well-defined and continuous in $v$ over the interval $(0,v_{\max})$. Under the condition of $\tfrac{\sigma}{R} < \sqrt{\mathscr{U}(M,n)}$, the continuity of $ \frac{\tau(v)-\mathscr{C}\bar{\tau}}{\mathscr{C}\bar{\mu}-\mu(v)}$ in $v$ ensures the existence of $\bar v\in(0,v_{\max})$ such that $\tfrac{\sigma^2}{R^2} \le \frac{\tau(\bar v)-\mathscr{C}\bar\tau}{\mathscr{C}\bar\mu-\mu(\bar v)}$, which, together with~\eqref{eq.objfunc} and~\eqref{eq.h1hconstant}, further implies that
\begin{align}\label{eq.h2h1}
    h(\bar v,\tfrac{\sigma}{R})
    =\tau(\bar v)+\mu(\bar v)\frac{\sigma^2}{R^2} \le \mathscr{C}\bar\tau+\mathscr{C}\bar\mu\frac{\sigma^2}{R^2} = \mathscr{C}h_1\!\left(\tfrac{\sigma}{R}\right).
\end{align}
Combining~\eqref{eq.h1hconstant}, \eqref{eq.h2hours}, and~\eqref{eq.h2h1}, we directly have that
\begin{align*}
h_{\mathrm{silver}}(\bar v,R,\sigma) \leq R^2 h(\bar v,\tfrac{\sigma}{R}) \leq R^2 \mathscr{C}h_1(\tfrac{\sigma}{R}) \leq \mathscr{C} h_{\rm constant}(R, \sigma).
\end{align*}

Finally, we show that $\mathscr{C} < 1$, which is equivalent to proving that, for all $M > 1$,
\begin{align*}
    \frac{(M + 1)^4\!\left(\sqrt{2M^2 - 2M + 1} - M\right)^2}
         {(M - 1)^4\!\left(2 + \sqrt{2M^2 - 2M + 1} - M\right)^2}
    < 1.
\end{align*}
Since the denominator in the inequality above is positive for any $M > 1$, it is equivalent to show that
\begin{align}\label{eq.Cleq1}
    (M + 1)^4\!\left(\sqrt{2M^2 - 2M + 1} - M\right)^2
    - (M - 1)^4\!\left(2 + \sqrt{2M^2 - 2M + 1} - M\right)^2 < 0.
\end{align}
Simplifying the left hand side of~\eqref{eq.Cleq1} yields
\begin{align*}
    -4\sqrt{1 - 2M + 2M^2}\,(1 - 4M + 10M^2 - 4M^3 + 5M^4)
    + 4(-1 + 7M - 14M^2 + 18M^3 - 9M^4 + 7M^5) < 0.
\end{align*}
For any $M > 1$, both polynomial terms are positive such as
\begin{align*}
    1 - 4M + 10M^2 - 4M^3 + 5M^4
    &= 5M^3(M-1) + M^3 + 10M(M-1) + 6M + 1 > 0,\\[3pt]
    -1 + 7M - 14M^2 + 18M^3 - 9M^4 + 7M^5
    &= (M^3(7M - 2) + 16M^2 + 2M + 9)(M - 1) + 8 > 0.
\end{align*}
Consequently, the inequality~\eqref{eq.Cleq1} is true if and only if
\begin{align*}
    (-1 + 7M - 14M^2 + 18M^3 - 9M^4 + 7M^5)^2
    - (1 - 2M + 2M^2)(1 - 4M + 10M^2 - 4M^3 + 5M^4)^2
    < 0,
\end{align*}
resulting in $-M(M - 1)^6\left(M^3 + 2M^2 + M + 4\right) \le 0$, which holds for all $M > 1$. Hence, $\mathscr{C} < 1$.
\end{proof}
\end{document}